\documentclass[reqno]{amsart}
\newcommand \datum {\magenta{June 12, 2021}}

\usepackage{amssymb,latexsym}
\usepackage{amsmath}
\usepackage{amscd}
\usepackage{graphicx}
 \usepackage{color}
\usepackage{enumerate}
\numberwithin{equation}{section}
\theoremstyle{plain}
 \newtheorem{theorem}{Theorem}[section]
 \newtheorem{lemma}[theorem]{Lemma}
 \newtheorem{proposition}[theorem]{Proposition}
 \newtheorem{corollary}[theorem]{Corollary}
\theoremstyle{definition}

\newcommand \tbf[1]  {\textbf{#1}} 
\newcommand \alg[1]  {\mathcal #1}
\newcommand \Ejoin[2] {E_{\textup{join}}(#1,#2)}
\newcommand \Emeet[2] {E_{\textup{meet}}(#1,#2)}
\newcommand \Nnul {\mathbb N_0}
\newcommand \Jir [1] {\textup J(#1)} 
\newcommand \length [1] {\textup{length}(#1)} 
\renewcommand \dim [1] {\textup {dim}_{\textup{ord}}(#1)}
\renewcommand \phi {\varphi}
\newcommand \width [1] {\textup {width}(#1)}
\newcommand \Nplu {\mathbb N^+}
\newcommand \set [1]{\{#1\}}

\newcommand \Dnfin[1] {\alg D(#1)}
\newcommand \Dncovzo [1] {\alg {D}_{\textup{01$\prec$}}(#1)}
\newcommand \covzo {cover-preserving $\set{0,1}$-homomorphism}
\newcommand \embcovzo {cover-preserving $\set{0,1}$-embedding}
\newcommand \Dall  {\alg D_{\textup{all}}}

\newcommand \ideal {\mathord\downarrow}
\newcommand \filter {\mathord\uparrow}
\newcommand \Ep [1] {E^+_{#1}}

\newcommand \cj [2] {#1[#2]}

\newcommand \ata [2] {a^{(#1)}_{#2}}
\newcommand \atom[1] {\ata{#1}1}
\newcommand \eqeqref [1] {\overset{(\ref{#1})}=}
\newcommand \diag [1] {\Delta_{#1}}
\newcommand \restrict[2] {#1\lceil_{#2}}
\newcommand \id [1] {\textup{id}_{#1}}
\newcommand \Hop  {\mathbf H}

\newcommand \magenta [1] {{\color{magenta}#1\color{black}}}

\begin{document}

\title[Absolute retracts for two classes of lattices]
{{Absolute retracts for finite distributive lattices and slim semimodular lattices}}

\author[G.\ Cz\'edli]{G\'abor Cz\'edli}
\email{czedli@math.u-szeged.hu}
\urladdr{http://www.math.u-szeged.hu/~czedli/}
\address{ Bolyai Institute, University of Szeged, Hungary 
}

\author[A.\ Molkhasi]{Ali Molkhasi}
\email{molkhasi@gmail.com , molkhasi@cfu.ac.ir}
\address{Department of Mathematics, Farhangian University of Iran, Tabriz}

\begin{abstract} 
{We describe the absolute retracts for the following classes of finite lattices: (1) slim semimodular lattices, (2) finite distributive lattices, and for each positive integer $n$, (3) at most $n$-dimensional finite distributive lattices. 
Although the singleton lattice is the only absolute retract
for the first class, this result has paved the way to some other classes. For the second class, we prove that the absolute retracts are exactly the finite boolean lattices; this generalizes a 1979 result of J. Schmid. 
For the third class, the absolute retracts are the finite boolean lattices of dimension at most $n$ and the direct products of $n$ nontrivial finite chains.
Also, we point out that in each of these classes, the algebraically closed lattices and the strongly algebraically closed lattices are the same as the absolute retracts.} 

{Slim (and necessarily planar) semimodular lattices were introduced by G.\ Gr\"atzer and E.\ Knapp in 2007, and  they have been intensively studied since then. 
Algebraically closed and strongly algebraically closed lattices have been investigated by J. Schmid and, in several papers, by  A. Molkhasi.} 
\end{abstract}

\thanks{This research of the first author was supported by the National Research, Development and Innovation Fund of Hungary under funding scheme K 134851.}



\date{\tbf{{\datum}.} 
\magenta{An earlier 8-page version was entitled \emph{``Is there an absolute retract for the class of slim semimodular lattices?''}} (Hint: check the author's website for possible updates)  }

\maketitle

\section{Introduction}\label{sect:intro}
Before formulating our targets and results  in Subsection~\ref{subsect-target}, we give a short historical overview.
The history leading to the present work belongs to {four} topics, which are surveyed in the following {four} subsections. 
According to our knowledge, {the first three of these four} topics have been studied independently so far; one of our goals is to find some connection among them.

\subsection{Strongly algebraically closed algebras {in categories of algebras}} 
By an \emph{equation} in an algebra $A$ we mean a formal expression
\begin{equation*}
p(a_1,\dots,a_m, x_1,\dots, x_n)\approx q(a_1,\dots,a_m, x_1,\dots, x_n)
\end{equation*}
where $m\in\Nnul=\set{0,1,2,\dots}$, $n\in\Nplu=\Nnul\setminus\set 0$,
$p$ and $q$ are $(m+n)$-ary terms (in the language of $A$), the elements $a_1,\dots,a_m$ belong to $A$ and they are called  \emph{parameters} (or \emph{coefficients}), and 
$x_1,\dots, x_n$ are the \emph{unknowns} of this equation. Although a single equation contains only finitely many unknowns, we allow infinite \emph{systems} (that is, sets) of equations and such a system can contain infinitely many unknowns. 

{By a \emph{category of algebras} we mean a concrete category $\alg X$ such that the objects of $\alg X$ are algebras of the same type,  every morphism of $\alg X$ is a homomorphism, and whenever $A_1$ and $A_2$  are isomorphic algebras such that $A_1$ belongs to $\alg X$, then so does $A_2$.
Note that there can be homomorphisms among the objects of $\alg X$ that are not morphisms of $\alg X$. 
If $\alg X$ happens to contain all homomorphisms among its objects as morphisms, then  $\alg X$ is a \emph{class of algebras $($with all homomorphisms$)$}; the parenthesized part of this term is often dropped in the literature.  Given a category $\alg X$ of algebras and objects $A,B$ in $\alg X$, we say that $B$ is an \emph{$\alg X$-extension} of $A$ if $A$ is a subalgebra of $B$ and, in addition, the map $\iota\colon A\to B$ defined by $x\mapsto x$ is a morphism in $\alg X$. (If $\alg X$ is a class of algebras with all homomorphisms and $A,B\in \alg X$, then ``extension" is the same as ''$\alg X$-extension".)}

{Note that the concept of ``$B$ is an $\alg X$-extension of $A$''
includes not only $A$ and $B$, but also the embedding  $\iota\colon A\to B$ defined by $x\mapsto x$. Therefore, when we speak of ``all $\alg X$-extensions of $A$'', then the meaning is  that all possible embeddings $\iota$ are considered. For example, if $A$ is the two-element chain in the class $\alg L$ of lattices with all homomorphisms, then $A$ has three essentially different $\alg L$-extensions  into a three-element chain. }

For {a category $\alg X$ of algebras and} an  algebra $A\in\alg X$, 
we say that $A$ is \emph{strongly algebraically closed in}  $\alg X$ if for every {$\alg X$-extension} $B\in \alg X$  of $A$ and for any system $\Sigma$ of equations with parameters taken from $A$, if $\Sigma$ has a solution in $B$, then it also has a solution in $A$. 
Following Schmid~\cite{schmid}, if we replace ``any system $\Sigma$'' by ``any finite system $\Sigma$'', then we obtain the concept of an \emph{algebraically closed algebra} $A$ in  $\alg X$. 
These two concepts have been studied by many authors;
restricting ourselves to lattice theory, we only mention
Schmid~\cite{schmid} and 
Molkhasi~\cite{molkhasi16,molkhasi18a,molkhasi18b,molkhasi20}.

\subsection{Absolute retracts} Given an algebra $B$ and a subalgebra $A$ of $B$, we say that $A$ is a \emph{retract} of $B$ if there exists a homomorphism $f\colon B\to A$ such that $f(a)=a$ for all $a\in A$. 
 The homomorphism $f$ in this definition is called a \emph{retraction map} or a \emph{retraction} for short.

Now let $A$ be {an algebra belonging to a category $\alg X$ of algebras. 
We say that $A$ is an \emph{absolute retract for}  $\alg X$ if for 
any $\alg X$-extension $B$ of $A$, there exists a retraction $B\to A$ among the morphisms of $\alg X$. 
Similarly, $A$ is an \emph{absolute $\Hop$-retract for}  $\alg X$ if for any $\alg X$-extension $B$ of $A$, there exists a retraction $f\colon B\to A$ (but $f$ need not be a morphism of $\alg X$). The letter $\Hop$ in this terminology comes from ``homomorphism''. Although an absolute $\Hop$-retract is not a purely category theoretical notion, it helps us to state some of our assertions in a stronger form. Note that if $A$ is an absolute retract for $\alg X$, then it is also and absolute $\Hop$-retract for $\alg X$. Observe that
\begin{equation}
\parbox{8.5cm}{for a class $\alg X$ of algebras with all homomorphisms, absolute $\Hop$-retracts and absolute  retracts are the same.}
\label{pbx:rmDhwhspKlQB}
\end{equation} 
}

Absolute retracts emerged first in topology, and they  appeared in classes of algebras as soon as 1946; see Reinhold~\cite{reinhold}.
There are powerful tools to deal with homomorphisms and, in particular, retractions in {several categories of lattices}; 
we will benefit from these tools in Sections \ref{sect-thm} {and \ref{sect:dist}}.

\subsection{Slim semimodular lattices}
{For a finite lattice $L$, let $\Jir L$ stand for the set of nonzero join-irreducible elements of $L$. Note that $\Jir L$
is a poset (i.e., partially ordered set) with respect to the order inherited from $L$.}
Following Cz\'edli and Schmidt \cite{czgschtJH}, we say that a lattice is \emph{slim} if it is finite and $\Jir L$  is the union of two chains. 
Note that slim lattices are planar; see Lemma 2.2 of Cz\'edli and Schmidt \cite{czgschtJH}. As usual, a lattice $L$ is (upper) \emph{semimodular} if we have $x\vee z\preceq y\vee z$ for any $x,y,z\in L$ such that 
$y$ covers or equals $x$ (in notation, $x\preceq y$).
Since the pioneering paper Gr\"atzer and Knapp \cite{gratzerknapp1},
recent years have witnessed a particularly intense activity in studying 
\emph{slim semimodular lattices}; see 
Cz\'edli~\cite{czgreprhomr,czgmatrix,czgtrajcolor,czganotesps,czgrectectdiag,czgqplanar,czgasymp,czgaxiombipart},
Cz\'edli, D\'ek\'any, Gyenizse and Kulin~\cite{czgdgyk},
Cz\'edli, D\'ek\'any, Ozsv\'art, Szak\'acs, Udvari~\cite{czg--nora},
Cz\'edli and Gr\"atzer~\cite{czgggltsta,czgggresections,czgginprepar},
Cz\'edli,  Gr\"atzer, and Lakser \cite{czggghlswing},  
Cz\'edli and Makay~\cite{czgmakay},
Cz\'edli and Schmidt~\cite{czgschtJH,czgschtvisual,czgschtcompser,czgschtpatchwork}, Gr\"atzer~\cite{ggonaresczg,ggswinglemma,ggtwocover,ggSPS8},  Gr\"atzer and Knapp \cite{gratzerknapp3}, and Gr\"atzer and Nation~\cite{gr-nation}.
For the impact of these lattices on (combinatorial) geometry, see {Adaricheva and Bolat~\cite{adaribolat},} Adaricheva and Cz\'edli~\cite{adariczg}, Cz\'edli~\cite{czgcircles}, and (the surveying) Section 2 of Cz\'edli and Kurusa~\cite{czgkurusa}, {and see their impact on lattice theory in Ranitovi\'c and Tepav\v cevi\'c \cite{andrejaranitovic,andrejamarijana}.}

\subsection{Finite and $n$-dimensional distributive lattices}
\label{subsect:Dfdim}
It is well known that a finite distributive lattice $D$ is determined by the poset $\Jir D$ up to isomorphism. 
Borrowing a definition from  Dushnik and Miller~\cite{dushnikmiller}, the \emph{order dimension}  of a poset $P=(P;\leq_P)$, denoted by $\dim P$, is the least number $n$ such that the relation $\leq_P$ is the intersection of $n$ linear orderings on $P$.  We know from Milner and Pouzet  \cite{milnerpouzet} that $\dim P$ is also the least number $n$ such that $P$ has an order embedding into the direct product of $n$ chains. 
The \emph{width} of a poset $P$ is defined to be the maximum size of an antichain in $P$; it will be denoted by $\width P$. 
By Dilworth \cite[Theorem  1.1]{dilworth}, a finite poset $P$ is of width $n$ if and only if $P$ is the union of $n$ (not necessarily disjoint) chains but not a union of fewer chains.
As it is pointed out in  the first paragraph of page 276 in Rabinovitch and Rival~\cite{rabinovitchrival}, it follows from Dilworth~\cite{dilworth} that 
\begin{equation}
\text{for a finite distributive lattice $D$, $\dim D=\width{\Jir D}$.}
\label{eqtxt:dimDwJD}
\end{equation}
If $\dim D=n$, then $D$ is said to be \emph{$n$-dimensional}.

\subsection{Targets and results}\label{subsect-target} 
First, we are going to prove the following easy proposition. By a finite algebra we mean a finite nonempty set equipped with \emph{finitely many} operations.

\begin{proposition}\label{prop}
If $A$ is an algebra in a {category} $\alg X$ of algebras, then the following two conditions are equivalent.
\begin{enumerate}
\item\label{propa} $A$ is strongly algebraically closed in $\alg X$.
\item\label{propb} $A$ is an absolute {$\Hop$-}retract for $\alg X$.  
\end{enumerate}
Furthermore, if $\alg X$ consists of finite algebras, then each of (\ref{propa}) and  (\ref{propb}) is equivalent to 
\begin{enumerate}\setcounter{enumi}{2}
\item\label{propc} $A$ is algebraically closed in $\alg X$.
\end{enumerate} 
\end{proposition}

This proposition will be proved in Section~\ref{sect:retr}.
Armed with Proposition \ref{prop}, we are going to prove the following result {in Section~\ref{sect-thm}}.

\begin{theorem}\label{thmsps}
Let $L$ be a slim semimodular lattice and let $\alg S$ denote the class of all slim semimodular lattices with all homomorphisms. Then the following four conditions are equivalent.
\begin{enumerate}
\item\label{thmnul} $L$ is algebraically closed in $\alg S$.
\item\label{thma} $L$ is strongly algebraically closed in $\alg S$.
\item\label{thmb} $L$ is an absolute retract for $\alg S$.
\item\label{thmd} $L$ is the one-element lattice, i.e., $|L|=1$.
\end{enumerate}
\end{theorem}

Since the singleton lattice does not look too exciting in itself, it is worth noting the following. First, we {know neither a really short proof of this theorem nor a proof without using some nontrivial tool from the theory of slim semimodular lattices}.   {Second,  Theorem~\ref{thmsps} together with Molkhasi~\cite{molkhasi16,molkhasi18a,molkhasi18b,molkhasi20}  and Schmid~\cite{schmid} have just motivated a related result with infinitely many absolute retracts for the class of slim semimodular lattices with less morphisms than here; see 
Cz\'edli \cite{czgpatchabsrectr}. 
Third and mainly, as it is explained in Subsection~\ref{subsect:nBtPrf},  Theorem~\ref{thmsps} and the tools needed to prove it have paved the way to  Theorem \ref{thmdst} below. To formulate it, let $\omega$ stand for the least infinite cardinal number, let $\Nplu:=\set{1,2,3,4,\dots}$,  and  
\begin{equation}
\parbox{7cm}{for $n\in\Nplu\cup\set{\omega}$, let $\Dnfin n$ denote  the class of \emph{finite} distributive lattices with order dimension at most $n$, with all homomorphisms.}
\label{txtDnfindef} 
\end{equation}
By a \emph{nontrivial} lattice we mean a lattice with more than one element.
}

\begin{theorem}[Main Theorem]\label{thmdst}
Let $n\in\Nplu\cup\set{\omega}$, see \eqref{txtDnfindef},  and let $D\in\Dnfin n$. Then the following four conditions are equivalent.
\begin{enumerate}
\item\label{thmdsta} $D$ is algebraically closed in $\Dnfin n$.
\item\label{thmdstb} $D$ is strongly algebraically closed in $\Dnfin n$.
\item\label{thmdstc} $D$ is an absolute retract for  $\Dnfin n$.
\item\label{thmdstd} $D$ is a boolean lattice or $D$ is the direct product of $n$ nontrivial finite chains. 
\end{enumerate} 
\end{theorem}

We are going to prove this Theorem in Section~\ref{sect:dist}.
Since  $\Dnfin \omega$  is \emph{the class of all finite distributive lattices}, $\Dnfin \omega=\bigcup_{n\in\Nplu} \Dnfin n$, and the direct product of $\omega$ many nontrivial chains cannot be finite, Theorem \ref{thmdst} clearly implies the following corollary.

\begin{corollary}\label{cor:D-omega} Let $D$ be a \emph{finite}  distributive lattice.  
Then the following four conditions are equivalent.
\begin{enumerate}
\item\label{cor:D-omegaa} $D$ is algebraically closed in the class $\Dnfin \omega$ of finite distributive lattices with all homomorphisms.
\item\label{cor:D-omegab} $D$ is strongly algebraically closed in $\Dnfin \omega$.
\item\label{cor:D-omegac} $D$ is an absolute retract for  $\Dnfin \omega$.
\item\label{cor:D-omegad} $D$ is a boolean lattice. 
\end{enumerate} 
\end{corollary}

The proofs of the following three corollaries are given in Section~\ref{sect:dist}; note that two of them follow partly from the proof of Theorem~\ref{thmdst} rather than from the theorem itself.

\begin{corollary}\label{cor:DfnVGs} For a \emph{finite}  distributive lattice $D$, the following four conditions are equivalent.
\begin{enumerate}
\item\label{cor:DfnVGsa} $D$ is algebraically closed in the class $\Dall$ of all (not necessarily finite) distributive lattices with all homomorphisms.
\item\label{cor:DfnVGsb} $D$ is strongly algebraically closed in $\Dall$.
\item\label{cor:DfnVGsc} $D$ is an absolute retract for  $\Dall$.
\item\label{cor:DfnVGsd} $D$ is a boolean lattice.
\end{enumerate} 
\end{corollary}

Note that while Schmid~\cite{schmid} only allows  lattice embeddings and homomorphisms that preserve 0 and 1 whenever they exist, there is no such restriction in the present paper. Therefore, even the \eqref{cor:DfnVGsd} $\Rightarrow$ \eqref{cor:DfnVGsa} implication in Corollary~\ref{cor:DfnVGs} 
is stronger than what Schmid~\cite{schmid} proves for a finite boolean lattice $D$. The classes $\Dnfin n$ for $n\in\Nplu$ have not occurred in this context previously. Let us emphasize that Corollary~\ref{cor:DfnVGs} does not describe the absolute retracts for $\Dall$; it describes only the finite absolute retracts for this class.

For finite lattices $A$ and $B$, a lattice homomorphism $f\colon A\to B$ is said to be a \emph{\covzo} if $f(0)=0$,  
$f(1)=1$, and for all $x,y\in A$ such that $x\prec y$, we have that $f(x)\prec f(y)$. 
Since any two maximal chains in a finite semimodular lattice are of the same length (this is the so-called \emph{Jordan--H\"older chain condition}), we easily obtain the following observation; see Lemma~\ref{lemmaKzTfPsz} for a bit more information. 
\begin{equation}
\parbox{7.8cm}{if $A$ and $B$ are finite semimodular lattices and there exists a \covzo{} $A\to B$, then $A$ and $B$ are of the same length.}
\label{pbx:lnPrsHm} 
\end{equation}
Note that distributive lattices, to which we will apply \eqref{pbx:lnPrsHm}, are semimodular.

For  $n\in\Nplu\cup\set\omega$, let $\Dncovzo n$ denote the category consisting of finite distribute lattices of order dimension at most $n$ as objects and \covzo{s} as morphisms. 
(So $\Dncovzo n$ has the same objects as $\Dnfin n$, but it has much less morphisms.)

\begin{corollary}\label{cor-kvsnLptv}
Let $n\in\Nplu\cup\set{\omega}$, and let 
$D\in\Dncovzo n$. Then the following five conditions are equivalent.
\begin{enumerate}
\item\label{cor-kvsnLptva} $D$ is algebraically closed in $\Dncovzo n$.
\item\label{cor-kvsnLptvb} $D$ is strongly algebraically closed in $\Dncovzo n$.
\item\label{cor-kvsnLptvc} $D$ is an absolute $\Hop$-retract for  $\Dncovzo n$.
\item\label{cor-kvsnLptvd} $D$ is an absolute retract for  $\Dncovzo n$.
\item\label{cor-kvsnLptve} $D$ is a boolean lattice or $D$ is the direct product of $n$ nontrivial finite chains.  
\end{enumerate} 
\end{corollary}

Corollary~\ref{cor-kvsnLptv} shows that we can disregards many morphisms from the categories  occurring in Theorem~\ref{thmdst} so that absolute retracts remain the same.  This is not at all so for the category $\alg S$ occurring in Theorem~\ref{thmsps}; see Cz\'edli~\cite{czgpatchabsrectr} for details.
 
In the following corollary, ``nontrivial"  means  ``non-singleton"; let us repeat that  planar lattices are finite by definition. 

\begin{corollary}\label{cor:krSzwzlvnlZvl}
If $D$ is a planar distributive lattice, then the following five conditions are equivalent.
\begin{enumerate}
\item\label{cor:krSzwzlvnlZvla} $D$ is an absolute retract for the class of planar distributive lattices with all homomorphisms.
\item\label{cor:krSzwzlvnlZvlb}  $D$ is an absolute $\Hop$-retract for the category  of planar distributive lattices with \covzo{}s as morphisms.
\item\label{cor:krSzwzlvnlZvlc}  $D$ is an absolute retract for the category  of planar distributive lattices with \covzo{}s as morphisms.
\item\label{cor:krSzwzlvnlZvld} $|D|\leq 2$ or $D$ is the direct product of two nontrivial finite chains.
\end{enumerate} 
\end{corollary}

Of course, Proposition~\ref{prop} is applicable for both  classes mentioned in parts \eqref{cor:krSzwzlvnlZvla} and \eqref{cor:krSzwzlvnlZvlb} of Corollary~\ref{cor:krSzwzlvnlZvl}, and so we could add $2\cdot 2= 4$ additional equivalent conditions to this corollary.

\color{black}

\section{Proving our proposition}\label{sect:retr}
To ease the notation, we give the proof only for lattices; the general proof would be practically the same.

\begin{proof}[Proof of Proposition~\ref{prop}] First, we deal with the implication \eqref{propa} $\Rightarrow$ \eqref{propb} and, if $\alg X$ consists of finite lattices, also with the implication \eqref{propc} $\Rightarrow$ \eqref{propb}.

Assume that $\alg X$ is a class of lattices, $A\in \alg X$, and either $A$ is strongly algebraically closed in $\alg X$ or 
$\alg X$ consists of finite lattices and $A$ is algebraically closed in $\alg X$.
Let $B\in \alg X$ be an $\alg X$-extension of  $A$. We need to show the existence of a retraction $f\colon B\to A$. We can assume that $A$ is a proper sublattice of $B$, because the identity map of $B$ 
would obviously be a $B\to A$ retraction if $A=B$. The elements of $A$ and those of $B\setminus A$ will be called \emph{old elements} and \emph{new} elements, respectively.
For each new element $b$, we take an unknown $x_b$. 
For each pair $(a,b)\in B\times B$ of elements such that at least one of $a$ and $b$ is  new, we define an equation 
$\Ejoin a b$ according to the following six rules.
\allowdisplaybreaks{
\begin{align}
\text{If $a$ is old, $b$ is new, and $a\vee b$ is old, then 
$\Ejoin a b$ is $a\vee x_b \approx a\vee b$.}\label{joinono}
\\
\text{If $a$ is new, $b$ is old, and $a\vee b$ is old, then 
$\Ejoin a b$ is $x_a\vee b \approx a\vee b$.}\label{joinnoo}
\\
\text{If $a$ and $b$ are new and $a\vee b$ is old, then 
$\Ejoin a b$ is $x_a\vee x_b \approx a\vee b$.}\label{joinnno}
\\
\text{If $a$ is old, $b$ and $a\vee b$ are new, then 
$\Ejoin a b$ is $a\vee x_b \approx x_{a\vee b}$.}\label{joinonn}
\\
\text{If $a$ and $a\vee b$ are new and $b$ is old, then 
$\Ejoin a b$ is $x_a\vee b \approx x_{a\vee b}$.}\label{joinnon}
\\
\text{If $a$, $b$, and $a\vee b$ are all new, then 
$\Ejoin a b$ is $x_a\vee x_b \approx x_{a\vee b}$.}\label{joinnnn}
\end{align}
}%
Analogously,  replacing $\vee$ by $\wedge$, we define
the equations $\Emeet a b$ for all $(a,b)\in B\times B$ such that at least one of $a$ and $b$ is a new element.
Let $\widehat E$ be the system of all equations we have defined so far. Note that if $\alg X$ consists of finite lattices, then $\widehat E$ is finite.

Clearly, $\widehat E$ has a solution in $B$. Indeed, we can let $x_b:=b$ for all new elements $b$ to obtain a solution of $\widehat E$. 
Since we have assumed that either $A$ is strongly algebraically closed in $\alg X$ or $\alg X$ consists of finite lattices and 
$A$ is algebraically closed in $\alg X$, it follows that $\widehat E$ also has a solution in $A$. This allows us to fix a solution of $\widehat E$ in $A$. That is, we can choose an element $u_b\in A$ for each new element $b$ such that  
the equations \eqref{joinono}--\eqref{joinnnn} turn into true equalities when the unknowns $x_b$, for $b\in B\setminus A$, are replaced by the elements $u_b$.

Next, consider the map
\begin{equation*}
f\colon B\to A,\text{ defined by }c\mapsto
\begin{cases}
c,&\text{if $c$ is an old element,}\cr
u_c,&\text{if $c$ is a new element.}
\end{cases}
\end{equation*}
We claim that $f$ is a retraction. Clearly, $f$ acts identically on $A$. So we need only to show that $f$ is a homomorphism. It suffices to verify that $f$ commutes with joins since the case of meets is analogous. If $a,b\in A$, then $a\vee b$ is also in $A$, and we have that $f(a)\vee f(b)=a\vee b=f(a\vee b)$, as required. If, say,
$a, a\vee b\in A$ and $b\in B\setminus A$, then \eqref{joinono} applies and we obtain that  $f(a)\vee f(b)= a\vee u_b= a \vee b=f(a\vee b)$, as required. If $a,b,a\vee b$ are all new, then we can use \eqref{joinnnn} to obtain that  $f(a)\vee f(b)=u_a\vee u_b=u_{a\vee b}=f(a\vee b)$, as required. The rest of the cases follow similarly from 
\eqref{joinnoo}--\eqref{joinnon}. Thus, we conclude that $f$ commutes with joins. We obtain analogously that it commutes with meets, whereby $f$ is a homomorphism. So $f$ is a retraction, proving that  \eqref{propa} $\Rightarrow$ \eqref{propb}
and, if $\alg X$ consists of finite lattices,  \eqref{propc} $\Rightarrow$ \eqref{propb}.

To prove the  implication, \eqref{propb} $\Rightarrow$ \eqref{propa}, assume that  $A\in \alg X$ is an absolute $\Hop$-retract for $\alg X$,  $B\in \alg X$ is an $\alg X$-extension of $A$, and 
a system $\widehat G$ of equations with constants taken from $A$ has a solution in $B$.

Let $x,y,z,\dots $ denote the unknowns occurring in $\widehat G$ (possibly, infinitely many), and let $b_x, b_y, b_z,\dots \in B$ form a solution of $\widehat G$. Since we have assumed that $A$ is an absolute $\Hop$-retract for $\alg X$, we can take a retraction $f\colon B\to A$. We define $d_x:=f(b_x)$, $d_y:=f(b_y)$, $d_z:=f(b_z)$, \dots; they are elements of $A$.
 Let $p(a_1,\dots,a_k, x,y,z, ...)=q(a_1,\dots,a_k, x,y,z, ...)$ be one of the equations of $\widehat G$; here $p$ and $q$ are lattice terms, the constants   $a_1,\dots, a_k$ are in $A$, and only finitely many unknowns occur in this equation, of course. Using that $f$ commutes with lattice terms and,  at $=^\ast$, using also that $b_x$, $b_y$, $b_z$, \dots form a solution of the equation in question, we obtain that 
\begin{align*}
p(a_1,\dots,a_k, d_x, d_y,d_z,\dots)
= p(f(a_1),\dots,f(a_k), f(b_x), f(b_y),f(b_z,)\dots)\cr
= f(p(a_1,\dots,a_k, b_x, b_y,b_z,\dots)) =^\ast  f(q(a_1,\dots,a_k, b_x, b_y,b_z,\dots))=\cr
q(f(a_1),\dots,f(a_k), f(b_x), f(b_y),f(b_z),\dots)=q(a_1,\dots,a_k, d_x, d_y,d_z,\dots).
\end{align*}
This shows that $d_x,d_y,d_z,\dots \in A$ form a solution of $\widehat G$ in $A$. Therefore, $A$ is strongly algebraically closed in $\alg X$, showing the validity of  \eqref{propb} $\Rightarrow$ \eqref{propa}.

Finally, the implication  \eqref{propa} $\Rightarrow$ \eqref{propc} is trivial,  completing the proof of Proposition~\ref{prop}.
\end{proof}

\section{{Proving} Theorem \ref{thmsps}}\label{sect-thm}
First, we recall briefly from Cz\'edli and Schmidt \cite{czgschtvisual} what we need to know about slim semimodular lattices. {Let us repeat that slim lattices are finite by definition; every lattice in this section is assumed to be \emph{finite}}.
For a slim semimodular lattice $L$, we always assume that a planar diagram of $L$ is fixed. 
A  cover-preserving four-element boolean sublattice of $L$ is called a \emph{$4$-cell}.
For $m,n\in\Nplu$, the direct product of an $(m+1)$-element chain and an $(n+1)$-element chain is called a \emph{grid} or, when we want to be more precise, an \emph{$m$-by-$n$ grid}; note that this grid has exactly $mn$ 4-cells.

We can add a \emph{fork} to a 4-cell of a slim semimodular lattice as it is shown in Figure 5 of \cite{czgschtvisual}; this is also shown here in Figure~\ref{figabsretr1}, where we have added a fork to the light-grey  4-cell of $S_7^{(1)}$ to obtain $S_7^{(2)}$, and in Figure~\ref{figllstr}, where we can obtain $R$ from the grid $G$ by adding a fork to the upper 4-cell of $G$.   \emph{Corners} are particular doubly irreducible elements on the boundary of $L$, see Figure 2 in \cite{czgschtvisual}, but we do not need their definition here. Instead of the exact definition of slim rectangular lattices, it suffices to know their characterization, which is given by (the last sentence of) Theorem 11 and  Lemma 22 in \cite{czgschtvisual} as follows:
\begin{equation}
\parbox{9cm}{$L$ is a
\emph{slim rectangular} lattice if and only if it can be obtained from a  grid by adding forks, one by one, in a finite (possibly zero) number of steps.}
\label{pbx:sRctGlr}
\end{equation}
We know from Lemma 21 of \cite{czgschtvisual} that 
\begin{equation}
\parbox{9cm}{a lattice $L$ is a slim semimodular lattice if and only if $|L|\leq 2$ or $L$ can be obtained  from a slim rectangular lattice by removing finitely many corners, one by one.}
\label{pbx:sLcRGns}
\end{equation}

\begin{figure}[ht]
\centerline
{\includegraphics[width=\textwidth]{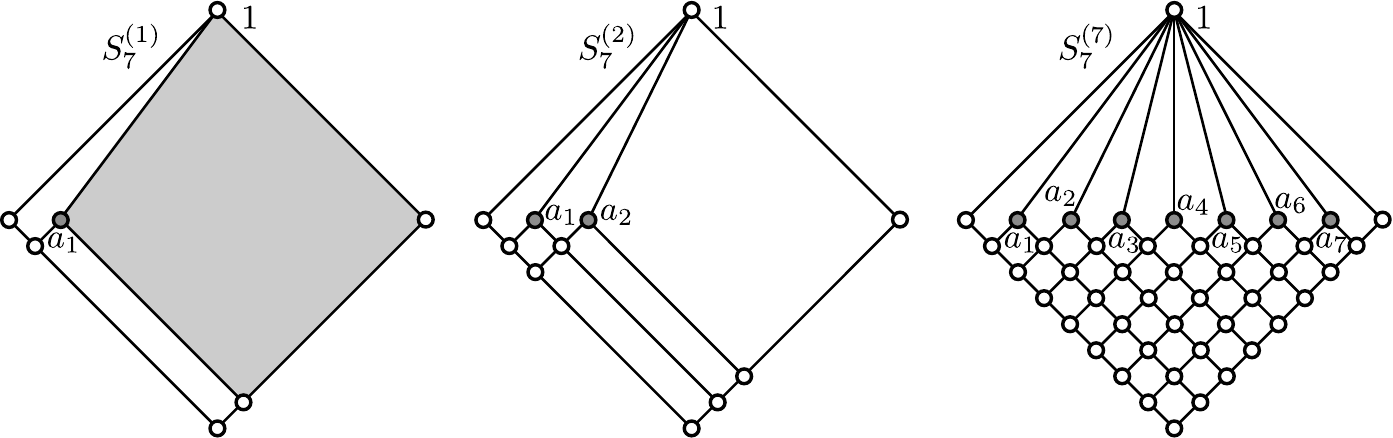}}      
\caption{$S_7^{(1)}$, $S_7^{(2)}$, and $S_7^{(7)}$   }\label{figabsretr1}
\end{figure}

\begin{proof}[Proof of Theorem \ref{thmsps}]
Since slim semimodular lattices are finite by definition, the equivalence of  \eqref{thmnul} and  \eqref{thma} follows trivially from Proposition~\ref{prop}. Also, Proposition~\ref{prop} yields the equivalence of \eqref{thma} and \eqref{thmb}.  Since the one-element lattice is an absolute retract for any class of lattices containing it, the implication \eqref{thmd} $\Rightarrow$ \eqref{thmb}  is trivial.

Thus, it suffices to prove the implication  \eqref{thmb} $\Rightarrow$ \eqref{thmd}.  To do so, it is sufficient to prove that whenever $L\in\alg S$ and $|L|\geq 2$, then $L$ is not an absolute retract for $\alg S$.  
So let $L$ be a slim semimodular lattice with at least two elements. By \eqref{pbx:sLcRGns} (or trivially if $|L|=2$), we can pick a slim rectangular lattice $R$ such that $L$ is a sublattice of $R$. It follows from \eqref{pbx:sRctGlr} that there exist $m,n\in\Nplu$ such that $R$ can be obtained from an $m$-by-$n$ grid $G$ by adding forks, one by one. Let $t\in\Nplu$ denote the smallest number such that $m+n+1\leq t$ and $|L|<t$. 

To present an example that helps the reader follow the proof, let $L$ be the 9-element slim semimodular lattice on the top left of  Figure \ref{figllstr}. For this $L$, we define  $R$ and $G$ by the top right diagram and the bottom right diagram of Figure \ref{figllstr}, respectively, and we have that $m=2$, $n=1$, and $t=10$.

\begin{figure}[ht]
\centerline
{\includegraphics[width=\textwidth]{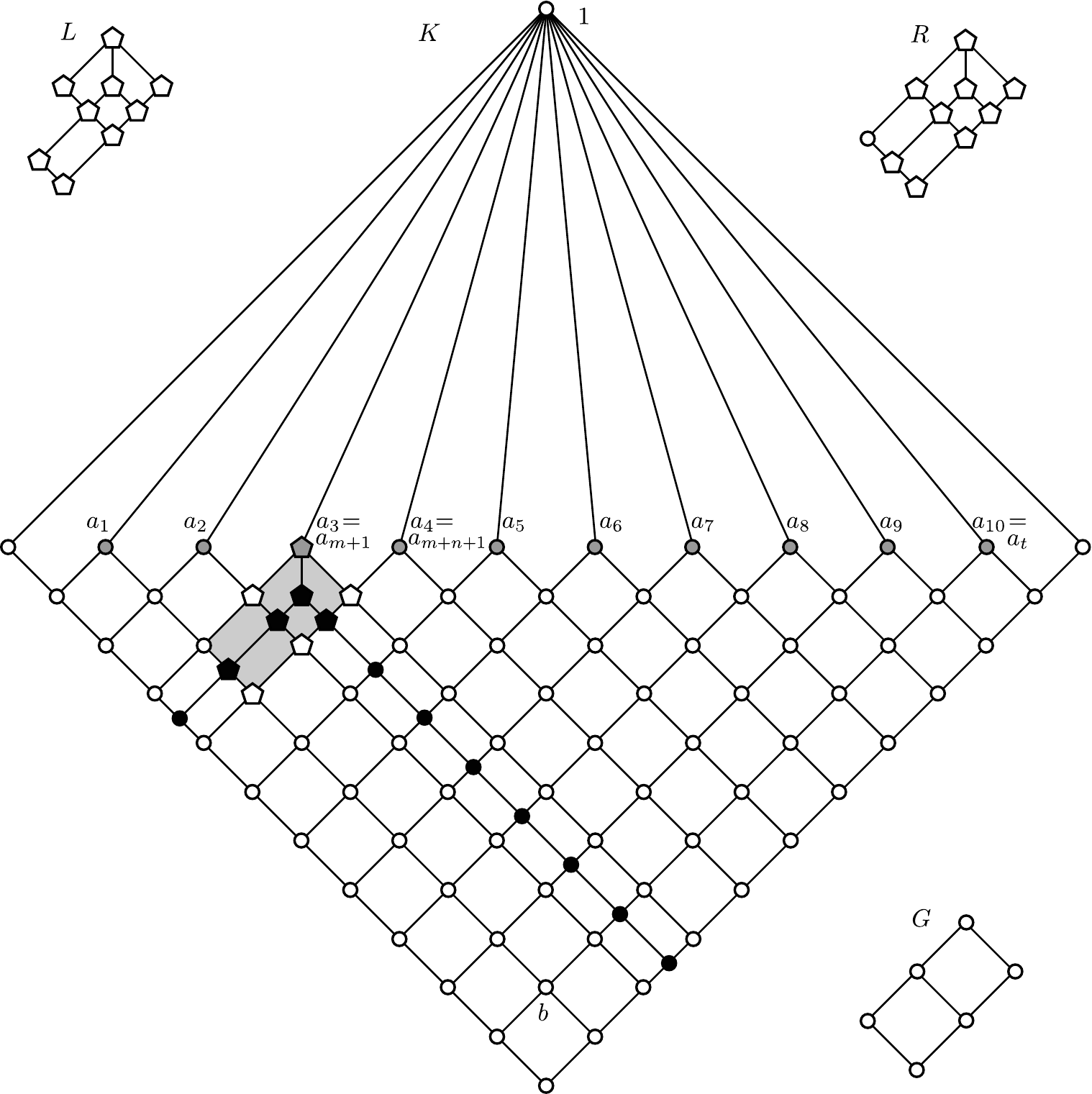}} 
\caption{Illustrating the proof of Theorem~\ref{thmsps}}\label{figllstr}
\end{figure}

We define the lattices $S_7^{(i)}$ for  $i\in\Nplu$ by induction as follows; see Figure~\ref{figabsretr1} for $i\in\set{1,2,7}$, and see the diagram in the middle of Figure~\ref{figllstr} for $i=10$ if we disregard the black-filled elements. (That is, $S_7^{(i)}=K\setminus\{$black-filled elements$\}$ in this diagram.)
Resuming the definition of the lattices $S_7^{(i)}$, we obtain $S_7^{(1)}$ by adding a fork to the only 4-cell of the four-element boolean lattice. From  $S_7^{(i)}$, we obtain $S_7^{(i+1)}$ by adding a fork to the rightmost 4-cell of $S_7^{(i)}$ that contains $1$, the largest element of $S_7^{(i)}$. 
(Note that we have also defined a fixed planar diagram of $S_7^{(i)}$  in this way.) The elements of $S_7^{(i)}$ (or those of a planar lattice diagram) not on the boundary of the diagram are called \emph{inner elements}. Let $a_1,a_2, \dots, a_i$ be the inner coatoms of $S_7^{(i)}$, listed from left to right. In our diagrams, they are grey-filled.
From now on, we only need  $S_7^{(t)}$. It follows from \eqref{pbx:sRctGlr} that $S_7^{(t)}$ is a slim semimodular (in fact, a slim rectangular) lattice. The meet $a_1\wedge\dots \wedge a_t$ of its inner coatoms will be denoted by $b$, as it is indicated in Figure~\ref{figllstr}.

Since $m+n+1\leq t$, the interval $[b,a_{m+1}]$ of $S_7^{(t)}$ includes an $m$-by-$n$ grid $G'$ with top element $a_{m+1}$. In our example, $G'$ is indicated by the light-grey area in the sense that $G'$ consists of those six elements of $S_7^{(t)}=S_7^{(10)}$ that are on the boundary of the light-grey rectangle. (Remember that $S_7^{(10)}= K\setminus\set{\text{black-filled elements}}$ in the middle of Figure \ref{figllstr}.)
Since the grids $G'$ and $G$  have the same {``sizes''}, they are isomorphic.  Thinking of  the diagrams,  we can even assume that $G'$ and $G$  are geometrically congruent. Hence, when we add forks to $G$ one by one in order to get $R$, we can simultaneously add forks to 
$G'$ in the same way  and, consequently, also to  $S_7^{(t)}$. In this way, we obtain a slim rectangular lattice $K$  from $S_7^{(t)}$; this follows from \eqref{pbx:sRctGlr}. Note that $K\in \alg S$. 
In the middle of Figure~\ref{figllstr}, $K$ consists of the empty-filled elements, the grey-filled elements, and the  black-filled elements. In $K$, the former interval $G'$ has become an interval isomorphic to $R$. But $R$ is an extension of $L$, whereby $K$ has a sublattice $L'$ such that $L'$ is isomorphic to $L$. In the middle of the figure, 
the elements of $L'$ are the pentagon-shaped larger elements. Note that the original inner coatoms $a_{1}, \dots, a_t$ are also inner coatoms of~$K$.

Next,  for the sake of contradiction, suppose that $L$ is an absolute retract for $\alg S$. Then so is $L'$ since $L'\cong L$. 
Since  $K\in\alg S$ and $L'$ is a sublattice of $K$,  there exists a retraction $f\colon K\to L'$. Let $\Theta:=\set{(x,y)\in K^2: 
f(x)=f(y)}$ be the kernel of $f$. Then $\Theta$ is a congruence of $K$ with exactly $|L'|$ blocks. But $t > |L|=|L'|$, whence there are distinct $i,j\in\set{1,\dots, t}$ such that $a_i$ and $a_j$ belong to the same $\Theta$-block. Hence, $(a_i,a_j)\in\Theta$, implying that $(a_i,1)=(a_i\vee a_i, a_j\vee a_i)\in \Theta$. Thus, the $\Theta$-block $1/\Theta$ of $1$ contains $a_i$. By Gr\"atzer's Swing Lemma, see his paper \cite{ggswinglemma} (alternatively, see   Cz\'edli, Gr\"atzer, and  Lakser \cite{czggghlswing} or Cz\'edli and Makay \cite{czgmakay} for secondary sources), $\set{a_1,\dots,a_t}\subseteq 1/\Theta$. Since congruence blocks are sublattices, $b=a_1\wedge\dots\wedge a_t \in 1/\Theta$. Therefore, using the facts  that  congruence blocks are \emph{convex} sublattices, 
$a_{m+1}\in 1/\Theta$, and $G'$ was originally a subinterval of $[b,a_{m+1}]$ in $S_7^{(t)}$, we obtain that $L'\subseteq [b,a_{m+1}]\subseteq 1/\Theta$ in the lattice $K$. Hence, for any $x,y\in L'$, we have that  $(x,y)\in \Theta$. Consequently, the definition of $\Theta$ and that of a retraction yield that, for any $x,y\in L'$,
$x=f(x)=f(y)=y$. Therefore, $|L|=|L'|=1$, which is a contradiction.
This contradiction implies that neither $L'$, nor $L$ is an absolute retract for $\alg S$,  completing the proof of Theorem~\ref{thmsps}.
\end{proof}

\section{Proving Theorem \ref{thmdst} and its corollaries}\label{sect:dist}

\subsection{Notes before the proof}\label{subsect:nBtPrf} 
This subsection is to enlighten the way from Theorem~\ref{thmsps} to Theorem~\ref{thmdst}. The reader is not expected to check the in-line statements in this subsection; what is needed will be proved or referenced in due course.

In the proof of Theorem \ref{thmsps},  forks play a crucial role. This raises the question what happens if forks are excluded from  \eqref{pbx:sRctGlr}. It follows from Cz\'edli and Schmidt~\cite[Lemma 15]{czgschtvisual} (and the proof of Corollary~\ref{cor:krSzwzlvnlZvl} here) that the lattices we obtain by means of \eqref{pbx:sRctGlr} and \eqref{pbx:sLcRGns} \emph{without} adding forks are exactly the members of $\Dnfin 2$.
But $\Dnfin 2$ is the class of \emph{distributive} slim  semimodular lattices. Hence, utilizing the theory of slim semimodular lattices, the particular case $n=2$ of Theorem ~\ref{thmdst} becomes available with little effort. Although this section is more ambitious by allowing $n\in\Nplu\cup\set{\omega}$, 
the ideas extracted from the theory of slim semimodular lattices and from the proof of Theorem~\ref{thmsps} 
have been decisive in reaching Theorem ~\ref{thmdst}.

\subsection{Auxiliary lemmas}
Unless otherwise explicitly stated, every lattice in this section is assumed to be \emph{finite}.
By an \emph{$n$-dimensional grid} we mean the direct product of  $n$ nontrivial (that is, non-singleton) finite chains. Clearly, the order dimension of an $n$-dimensional grid is $n$. 
2-dimensional grids are simply called grids in Section \ref{sect-thm}. 
For an $n$-dimensional grid $G$ and a maximal element $a\in\Jir G$, the principal ideal $\ideal a$ is a nontrivial chain. Chains of this form will be called the \emph{canonical chains} of $G$.
The following lemma follows  trivially from the fact that in a direct product of finitely many finite chains we compute componentwise. 

\begin{lemma}\label{lemma:cnChgR} If $n\in\Nplu$ and $G$ is an $n$-dimensional grid, then the following assertions hold.
\begin{enumerate}
\item\label{lemma:cnChgRa} $G$ has  exactly $n$ canonical chains; in the rest of the lemma, they will be denoted by $C_1$, \dots, $C_n$.
\item\label{lemma:cnChgRb}  Each element $x$ of $G$ can uniquely be written in the \emph{canonical form}
\begin{equation}
 \parbox{9.2cm}{$x=\cj x1\vee\dots\vee \cj xn\,\,$ where $\cj x1:=x\wedge 1_{C_1}\in C_1$, \dots, $\cj xn:=x\wedge 1_{C_n}\in C_n$; the elements $\cj x1$,\dots,$\cj xn$ are called the \emph{canonical joinands} of $x$.}
\label{eq:cFnfRmsG}
\end{equation}
\item\label{lemma:cnChgRc} For each $i\in\set{1,\dots,n}$, the map $\pi_i\colon G\to C_i$ defined by $x\mapsto \cj xi$ is a surjective homomorphism. 
\item\label{lemma:cnChgRd} The map $G\to C_1\times\dots\times C_n$ defined by $x\mapsto(\cj x1,\dots,\cj  xn)$ is a lattice isomorphism. 
\end{enumerate}
\end{lemma}

The notation $\cj x1$, \dots, $\cj xn$ will frequently be used, provided the canonical chains of an $n$-dimensional grid  are fixed. The map $\pi_i$ above is often called the \emph{$i$-th projection}.
Note that, for an $n$-dimensional grid $G$,   $\Jir G$ is the disjoint union of $C_1\setminus\set 0$, \dots, $C_n\setminus\set 0$. Thus, the \emph{set} $\set{C_1,\dots,C_n}$ of the  canonical chains is uniquely determined, and only the \emph{order} of these chains needs fixing.   
We also need the following lemma; the sublattices of a chain are called \emph{subchains}.

\begin{lemma}\label{lemma:gRd}
Assume that $n\in\Nplu$,  $L$ and $K$ are $n$-dimensional grids, and $L$ is a sublattice of $K$. Then there are nontrivial subchains $E_1$, \dots, $E_n$ of the canonical chains $C_1$, \dots, $C_n$ of $K$, respectively, such that 
\begin{equation}
L=\set{x\in K: \cj x1\in E_1,\dots, \cj x n\in E_n}.
\label{eq:dcEdjKwmC}
\end{equation} 
\end{lemma}

The visual meaning of Lemma \ref{lemma:gRd} is that an $n$-dimensional grid cannot be embedded into another $n$-dimensional grid in a ``skew way''.

\begin{proof}[Proof of Lemma~\ref{lemma:gRd}] Assume that $n\in\Nplu$,  $L$ and $K$ are $n$-dimensional grids, and $L$ is a sublattice of $K$. Then  there are 
integers $t_1\geq 2$, \dots, $t_n\geq 2$ and chains 
$H_i=\set{0,1,\dots, t_i-1}$ (with the natural ordering of integer numbers) such that we can pick an isomorphism 
 $\phi\colon H_1\times\dots\times H_n\to L$. 
The canonical chains of $K$ will be denoted by $C_1$, \dots, $C_n$.
The least element of  $H_1\times\dots\times H_n$ and that of $L$ are $\vec 0:=(0,\dots,0)$ and    $0_L=\phi(\vec 0)$, respectively. For $(i_1,\dots,i_n)\in H_1\times\dots\times H_n$, we write $\phi(i_1,\dots,i_n)$ rather than the more precise $\phi((i_1,\dots,i_n))$. 
For $j\in\set{1,\dots,n}$ and $i\in H_j\setminus\set 0$, we are going to use the notation
\begin{equation}
\ata j i :=(\,\underbrace{0,\,\,\dots,\,\,0}_{j-1\text{ zeros}}\,,\,\,i,\,\,\underbrace{0,\,\,\dots,\,\,0}_{n-j\text{ zeros}}\,),
\label{eq:TmBjblChn}
\end{equation}
Clearly, 
\begin{equation}
\Jir{H_1\times\dots\times H_n}=\set{\ata j i: j\in\set{1,\dots,n}\text{ and }i\in H_i\setminus\set 0}.
\end{equation}
It is also clear that the atoms of $H_1\times\dots\times H_n$ are 
$\atom 1$, \dots, $\atom n$. 
With the notation given in \eqref{eq:cFnfRmsG},
for $j\in \set{1,\dots,n}$ we let 
\begin{equation}
I_j:=\set{i\in\set{1,\dots,n}: \cj{\phi(\atom j)}i > \cj{0_L}i }.
\label{eq:NshMdhzDfdPh}
\end{equation}
Since $\atom j>\vec 0$ and $\phi$ is an isomorphism, $I_j\neq\emptyset$. 
We claim  that
\begin{equation}
\text{if $j\neq k\in\set{1,\dots,n}$, then $I_j\cap I_k=\emptyset$.}
\label{eq:szTrvspr}
\end{equation}
For the sake of contradiction, suppose that $j\neq k$ but $i\in I_j\cap I_k$. Then $\cj{\phi(\atom j)}i > \cj{0_L}i$ and $\cj{\phi(\atom k)}i > \cj{0_L}i$. Since $j$ and $k$   play a symmetrical role and the elements $\cj{\phi(\atom j)}i$ and $\cj{\phi(\atom k)}i$ belonging to the same canonical chain $C_i$ of $K$ are comparable, we can assume that  
$\cj {0_L}i < \cj{\phi(\atom j)}i \leq \cj{\phi(\atom k)}i$.
Hence, using Lemma~\ref{lemma:cnChgR}\eqref{lemma:cnChgRc},
\begin{align*}
\cj{\phi(\atom j)}i 
&=  \cj{\phi(\atom j)}i \wedge \cj{\phi(\atom k)}i  =
 \cj{\bigl(\phi(\atom j)\wedge \phi(\atom k)\bigr)}i \cr
&=\cj{\phi(\atom j\wedge \atom k)}i=\cj{\phi(\vec 0)}i=\cj{0_L}i,
\end{align*}
contradicting \eqref{eq:NshMdhzDfdPh} and proving \eqref{eq:szTrvspr}.
Using that $I_1$, \dots, $I_n$ are nonempty subsets of the finite set $\set{1,\dots, n}$ and  they are pairwise disjoint by  \eqref{eq:szTrvspr},  we have that 
\[n\leq |I_1|+\dots+ |I_n|=|I_1\cup\dots\cup I_n|\leq |\set{1,\dots, n}|=n. 
\]
Hence, none of the $I_1$, \dots, $I_j$ can have more than one element, and we obtain that  $|I_1|=\dots=|I_n|=1$. Therefore, after changing the order of the direct factors in $H_1\times\dots\times H_n$ and so also the order of the atoms $\atom 1$, \dots, $\atom n$ if necessary, we can write that $I_1=\set 1$, \dots, $I_n=\set n$. 
This means that, for all $j,k\in\set{1,\dots,n}$,
\begin{equation}
\cj{\phi(\ata j1)}k\geq \cj{0_L}k, \text{ and }
\cj{\phi(\ata j1)}k > \cj{0_L}k \iff k=j.  
\label{eq:cSkmpSzkDfcRsb}
\end{equation}

Next, we generalize \eqref{eq:cSkmpSzkDfcRsb} by claiming that for $j,k\in\set{1,\dots,n}$ and $i\in H_j\setminus\set 0$, 
\begin{equation}
\cj{\phi(\ata j i)}k\geq \cj{0_L}k, \text{ and }
\cj{\phi(\ata j i)}k > \cj{0_L}k \iff k=j.  
\label{eq:rSmnVmsKrp}
\end{equation}
To prove this, we can assume that $i>1$ since otherwise \eqref{eq:cSkmpSzkDfcRsb} applies. 
Using  \eqref{eq:cSkmpSzkDfcRsb} together with the fact that  $\pi_k$ and $\pi_j$ defined in Lemma~\ref{lemma:cnChgR}\eqref{lemma:cnChgRc} are  order-preserving, we obtain that 
$\cj{\phi(\ata ji)}k\geq \cj{\phi(\ata j1)}k\geq \cj{0_L}k$ for all $k\in\set{1,\dots,n}$, as required, and $\cj{\phi(\ata ji)}j\geq \cj{\phi(\ata j1)}j > \cj{0_L}j$. So all we need to show is that $\cj{\phi(\ata ji)}k > \cj{0_L}k$ is impossible if $k\neq j$. For the sake of contradiction, suppose that $k\neq j$, $k,j\in\set{1,\dots,n}$, and $\cj{\phi(\ata ji)}k > \cj{0_L}k$. 
We also have that $\cj{\phi(\ata ki)}k > \cj{0_L}k$ since
$\pi_k$ is order-preserving and $\cj{\phi(\ata k1)}k > \cj{0_L}k$ by \eqref{eq:cSkmpSzkDfcRsb}. Belonging to the same canonical chain of $K$, the elements $\cj{\phi(\ata ji)}k$ and $\cj{\phi(\ata ki)}k$ are comparable, whence their meet is one of the meetands. Thus, $\cj{\phi(\ata ji)}k\wedge \cj{\phi(\ata ki)}k > \cj{0_L}k$.  Hence, using that $\phi$ and $\pi_k$ are homomorphisms and $\ata ji \wedge \ata k i=\vec 0$, we obtain that 
\begin{align*}
\cj{0_L}k &< \cj{\phi(\ata j i)}k\wedge \cj{\phi(\ata k i)}k =
 \cj{\bigl(\phi(\ata j i) \wedge \phi(\ata k i)\bigr)}k \cr
&=  \cj{\phi(\ata j i \wedge \ata k i)}k =  \cj{\phi(\vec 0)}k 
=\cj{0_L}k,
\end{align*}
which is a contradiction proving \eqref{eq:rSmnVmsKrp}.

Next, after extending  the notation given in  \eqref{eq:TmBjblChn} by letting $\ata j 0:=\vec 0$ for $j\in\set{1,\dots,n}$, we have that 
\begin{equation}
\cj{\phi(\ata k i )}k\geq \cj{0_L}k \,\,
\text{ for all }k\in\set{1,\dots,n}\text{ and }i\in H_k
\label{eq:tdBlrbKlR}
\end{equation}
since $\ata k i\geq \ata k 0=\vec 0$,  $\phi$ and $\pi_k$ are order-preserving maps, and $0_L=\phi(\vec 0)$. 
For $j\in\set{1,\dots,n}$, we define
\begin{equation}
E_j:=\set{ \cj{\phi(\ata j i )}j:  i\in H_j}.
\label{eq:Eunderscorej}
\end{equation}
By \eqref{eq:cFnfRmsG}, $E_j\subseteq C_j$, that is, $E_j$ is a subchain of $C_j$ for all $j\in\set{1,\dots,n}$.  
We are going to show that  these $E_j$ satisfy \eqref{eq:dcEdjKwmC}.

First, assume that $x\in K$ is of the form
$x=\cj x 1\vee\dots\vee \cj x n$ such that $\cj x j\in E_j$ for all $j\in\set{1,\dots, n}$. Then, for each $j\in\set{1,\dots,n}$, there is an $i(j)\in H_j$ such that $\cj x j =  \cj{\phi(\ata j {i(j)} )}j$. Using what we already have, let us compute:
\allowdisplaybreaks{
\begin{align}
x&=\cj x 1\vee\dots\vee \cj x n=\cj{\phi(\ata 1 {i(1)} )}1\vee\dots\vee \cj{\phi(\ata n {i(n)} )}n 
\label{eq:fPldMsTa}
\\
&\eqeqref{eq:tdBlrbKlR} \cj{\phi(\ata 1 {i(1)})}1   
\vee\dots\vee \cj{\phi(\ata n {i(n)} )}n  \vee \cj{0_L}1 \vee\dots \vee  \cj{0_L}n 
\cr
&\eqeqref{eq:rSmnVmsKrp}
\cj{\phi(\ata 1 {i(1)})}1   
\vee\dots\vee \cj{\phi(\ata n {i(n)} )}n  
\cr
&\phantom{m m i}\vee
\cj{\phi(\ata 1 {i(1)})}2\vee\dots\vee \cj{\phi(\ata 1 {i(1)})}n  
\cr
&\phantom{m m i}\vee\dots\vee 
\cj{\phi(\ata n {i(n)}}1\vee\dots\vee \cj{\phi(\ata n {i(n)})}{n-1} 
\cr
&\eqeqref{eq:cFnfRmsG}
\phi(\ata 1 {i(1)})\vee \dots\vee \phi(\ata n {i(n)})
=\phi(  \ata 1 {i(1)}\vee \dots\vee \ata n {i(n)} )
\label{eq:fPldMsTb}
\end{align}
}%
Since $\phi(  \ata 1 {i(1)}\vee \dots\vee \ata n {i(n)} )
\in\phi(H_1\times\dots\times H_n)=L$, the computation from \eqref{eq:fPldMsTa} to \eqref{eq:fPldMsTb} shows the  ``$\supseteq$"  part of \eqref{eq:dcEdjKwmC}.

To show the reverse inclusion, assume that $x\in L$. Applying Lemma \ref{lemma:cnChgR}\eqref{lemma:cnChgRb} to the $\phi$-preimage of $x$, we obtain the existence of $i(1),\dots,i(n)$ such that $x=\phi(  \ata 1 {i(1)}\vee \dots\vee \ata n {i(n)} )$. 
Reading the computation from \eqref{eq:fPldMsTb} to \eqref{eq:fPldMsTa} upward, it follows that  $x=\cj{\phi(\ata 1 {i(1)})}1\vee\dots\vee \cj{\phi(\ata n {i(n)})}n$. By the uniqueness part of Lemma \ref{lemma:cnChgR}\eqref{lemma:cnChgRb}, 
$\cj x 1= \cj{\phi(\ata 1 {i(1)} )}1$, \dots, $\cj x n=\cj{\phi(\ata n {i(n)})}n$. Combining this with \eqref{eq:Eunderscorej}, 
we have that $\cj x 1\in E_1$, \dots, $\cj x n\in E_n$. This yields the ``$\subseteq$'' inclusion for  \eqref{eq:dcEdjKwmC} and completes the proof of Lemma~\ref{lemma:gRd}.
\end{proof}

The following easy lemma sheds more light on the categories $\Dncovzo n$, $n\in\Nplu\cup\set\omega$. The \emph{length} of a lattice $M$ is denoted by $\length M$; for definition (in the finite case) we mention that if $C$ is a maximum-sized chain in $M$, then $\length M +1 = |C|$.

\begin{lemma}\label{lemmaKzTfPsz} Assume that $K, L$ are finite semimodular lattices (in particular, finite distributive lattices) and $f\colon K\to L$ is a map. Then the following two assertions hold.
\begin{enumerate}
\item\label{lemmaKzTfPsza}  If  $f$ is a \covzo, then $f$ is a \embcovzo{} and $\length K=\length L$.
\item\label{lemmaKzTfPszb}  If $f$ is a lattice embedding and $\length K=\length L$, then $f$ is a \covzo{}.
\item\label{lemmaKzTfPszc}  If $L$ is a sublattice of $K$ such that the map $L\to K$ defined by $x\mapsto x$ is a \embcovzo{} and $f$ is a retraction, then $f$ is a lattice isomorphism (and, in particular, $f$ is also a \embcovzo). 
\end{enumerate}
\end{lemma}

\begin{proof} First, recall the following concept. A sublattice $S$ of a lattice $M$ is a \emph{congruence-determining sublattice} of $M$ if any congruence $\alpha$ of $M$ is uniquely determined by its restriction $\restrict \alpha S:=\set{(x,y)\in S^2: (x,y)\in\alpha}$. By Gr\"atzer and Nation~\cite{gr-nation},
\begin{equation}
\parbox{7.2cm}{every maximal chain of a finite semimodular lattice is a congruence-determining sublattice.}
\label{pbx:gRnNt}
\end{equation}

To prove part \eqref{lemmaKzTfPsza}, let $f\colon K\to L$ be a \covzo.  We know from \eqref{pbx:lnPrsHm} that $\length K=\length L$. Let $\Theta:=\set{(x,y)\in K^2: f(x)=f(y)}$ be the kernel of $f$, and take a maximal chain $C$ in  $K$. For $c,d\in C$ such that $c\prec d$, we have that $(c,d)\notin\Theta$ since  $f(c)\prec f(d)$. Hence, using that the blocks of $\restrict\Theta C$ are convex sublattices of $C$, 
it follows that $\restrict \Theta C=\diag C$. Applying \eqref{pbx:gRnNt}, we have that $\Theta=\diag K$. Thus, $f$ is injective, proving part \eqref{lemmaKzTfPsza}.

We prove part \eqref{lemmaKzTfPszb} by way of contradiction. Suppose that in spite of the assumptions, $f$ is not cover-preserving. Pick $a,b\in K$ such that $a\prec b$ but $f(a)\not\prec f(b)$. The injectivity of $f$ rules out that $f(a)=f(b)$. Hence, the interval $[f(a),f(b)]$ is of length at least 2.  Extend $\set{a,b}$ to a maximal chain $C=\set{0=c_0,c_1,\dots, c_k=1}$ of $K$ such that $a=c_{i-1}$, $b=c_i$, and $c_0\prec c_1\prec\dots\prec c_k$. By the Jordan--H\"older chain condition, $k=\length K$. 
Using the injectivity of $\phi$ again and the fact that $\phi$ is order-preserving,  $\length{[f(c_{j-1}), f(c_j)]}\geq 1$ for all $j\in\set{1,\dots, k}$. So the summands in 
\begin{equation}
\length K \geq \sum_{j=1}^k \length{[f(c_{j-1}), f(c_j)]}
\label{eqszGhsplBBszlHk} 
\end{equation} 
are positive integers but the $i$-th summand is at least two. 
Therefore, this sum and $\length K$ are at least $k+1$, which is a contradiction completing the proof of 
part  \eqref{lemmaKzTfPszb}. 

Next, to prove part \eqref{lemmaKzTfPszc}, observe that $0_L=0_K$ and $1_L=1_K$. Hence, since $f$ is a retraction, $f(0_K)=0_L$ and $f(1_K)=1_L$, as required. We are going to show that whenever $a\prec b$ in $K$, then $f(a)\prec f(b)$ in $L$. For the sake of contradiction, suppose that $a\prec b$ in $K$ but $f(a)\not\prec f(b)$ in $L$. Then there are two cases (since $f$ is order-preserving): either we have that $f(a)=f(b)$, or $f(a)<f(b)$ and the length of the interval $[f(a),f(b)]$ is at least 2. 
For each of these two cases, 
let $\Theta$ denote the kernel $\set{(x,y)\in K^2: f(x)=f(y)}$ of $f$, and let $U=\set{0=u_0,u_1,\dots, u_k=1}$ be a maximal chain of $L$. It is also a maximal chain of $K$ since the embedding $L\to K$ defined by $x\mapsto x$ is a \covzo.
 We know from the Jordan--H\"older chain condition, $k=\length L$. 

First, we deal with the first case, $f(a)=f(b)$. Then $(a,b)\in\Theta$ shows that $\Theta\neq \diag K$. 
We have that  $\restrict\Theta U\neq \diag U$ since \eqref{pbx:gRnNt} applies. Using that the blocks of $\restrict\Theta U$ are convex sublattices of $U$, it follows that $(u_{i-1},u_i)\in \restrict\Theta U$ for some $i\in\set{1,\dots,n}$. This means that $f(u_{i-1})=f(u_i)$. 
This equality leads to a contradiction since $f$ is a retraction and so  $u_{i-1}=f(u_{i-1})=f(u_{i})=u_{i}$. 
Since the only conditions tailored to $a$ and $b$ were $a\prec b$ and $f(a)=f(b)$, we have also obtained that 
\begin{equation}
\text{if $a'\prec b'$ in $K$, then $f(a')\neq f(b')$.}
\label{eq:pCpszLklChL}
\end{equation}

Next, we focus on the case $f(a)<f(b)$ and $\length{[f(a),f(b)]}\geq 2$. As in the proof of part  \eqref{lemmaKzTfPszb}, we can extend $\set{a,b}$ to a maximal chain $C$ of $K$. 
Since $U$ is also a maximal chain of $K$,  the Jordan--H\"older chain condition gives that  $\length C=\length K=\length U=k$. This allows us to write that  $C=\set{0=c_0,c_1,\dots, c_k=1}$ where   $a=c_{i-1}$ and $b=c_i$ for some $i\in\set{1,\dots,k}$, and $c_0\prec c_1\prec\dots\prec c_k$. By the Jordan--H\"older chain condition, \eqref{eqszGhsplBBszlHk} is still valid. 
Each summand  in \eqref{eqszGhsplBBszlHk} is at least 1 by 
 \eqref{eq:pCpszLklChL}, but the $i$-th summand is 
$\length{[f(c_{i-1}),f(c_i)]}=\length{[f(a),f(b)]}\geq 2$. Hence, $k=\length K \geq k+1$, which is a contradiction again. 
In this way, we have shown that $f$ is a \covzo. 

Applying the already proven part \eqref{lemmaKzTfPsza} of  Lemma~\ref{lemmaKzTfPsz}, we obtain that $f$ is a \embcovzo. 
This yields that $|K|\leq |L|$. But we also have that $|L|\leq |K|$ since $L$ is a sublattice of $K$. Thus, $|K|=|L|$, whence 
the embedding $f$ is a lattice isomorphism since these lattices are finite. 
This completes the proof of part \eqref{lemmaKzTfPszc}
and that of Lemma~\ref{lemmaKzTfPsz}.
\end{proof}

\begin{lemma}\label{lemma:nLskPcjktNLTn}
If $L$ is a nontrivial finite distributive lattice with  order dimension $n\in\Nplu$, then there is a \embcovzo{} of $L$ into an $n$-dimensional grid $G$. 
\end{lemma}

\begin{proof}[Proof of Lemma \ref{lemma:nLskPcjktNLTn}] By \eqref{eqtxt:dimDwJD}, $\width{\Jir L}=n$. 
It follows from Dilworth \cite[Theorem  1.1]{dilworth}, mentioned already in Subsection \ref{subsect:Dfdim}, that there are chains $C_1$, \dots, $C_n$ in $\Jir L$ such that 
$\Jir L=C_1\cup\dots\cup C_n$. We define $E_1$, \dots, $E_n$ by induction as follows:
\[
E_1:=C_1 \text{ and, for $i\in\set{2,\dots,n}$, }  E_i:=C_i\setminus(C_1\cup\dots\cup C_{i-1}).
\]
We show by an easy induction that
\begin{equation}
\parbox{7.6cm}{for $i\in\set{1,\dots,n}$,
$E_1\cup\dots\cup E_i=C_1\cup\dots\cup C_i$,  and the sets 
$E_1$, \dots, $E_i$  are pairwise disjoint.}
\label{pbx:szWrknCSkBrGlnD}
\end{equation}
Since this is trivial for $i=1$, assume that $i\in\set{2,\dots, n}$ and   \eqref{pbx:szWrknCSkBrGlnD} holds for $i-1$. Then 
$E_1\cup \dots \cup E_{i-1}\cup E_i=  C_1\cup \dots \cup C_{i-1}\cup (C_i\setminus(C_1\cup\dots\cup C_{i-1}))=C_1\cup\dots\cup C_i$
shows the equality in \eqref{pbx:szWrknCSkBrGlnD} for $i$. The sets $E_1$, \dots, $E_{i-1}$ are pairwise disjoint by the induction hypothesis, while $E_i$ is disjoint from them because of $ E_i:=C_i\setminus(C_1\cup\dots\cup C_{i-1}) = C_i\setminus(E_1\cup\dots\cup E_{i-1})$.
This shows the validity of \eqref{pbx:szWrknCSkBrGlnD}. 

Next, with $\Ep i:=E_i\cup\set 0$ for $i\in \set{1,\dots,n}$ and $0=0_L\notin E_i$, we define $G:=\Ep 1\times\dots \times \Ep n$. 
Since  $L$ is nontrivial (that is, $|L|>1$), we have that $|\Ep i|\geq 2$ and so $G$ is an $n$-dimensional grid. Clearly, $|\Jir G|=|E_1|+\dots+|E_n|$. This equality and \eqref{pbx:szWrknCSkBrGlnD} give that 
$|\Jir G|=|E_1\cup\dots\cup E_n|=|C_1\cup\dots\cup C_n|=|\Jir L|$.
We know from the folklore or from Gr\"atzer~\cite[Corollary 112]{r:Gr-LTFound} that 
\begin{align}
&\parbox{7.2cm}{the length of a finite distributive lattice equals the number of its join-irreducible elements,}\label{pbx:lJrsBlTlgR}\\
&\text{whereby $G$ and $L$ are of the same length.}
\label{eqtxt:sHmlFksRb}
\end{align}

For  $x\in L$ and $i\in\set{1,\dots n}$, let $x_i$ stand for the largest element of $\Ep i\cap \ideal x$; this makes sense since $\Ep i$ is a chain of $L$ and $0\in \Ep i\cap \ideal x$ shows that $\Ep i\cap \ideal x\neq \emptyset$. We are going to show that 
\begin{equation}
\parbox{6.5cm}{the map $\phi\colon L\to G$ defined by the rule $x\mapsto (x_1,\dots,x_n)$ is a lattice embedding.}
\label{pbx:zgmSzcskHt}
\end{equation} 
To prove \eqref{pbx:zgmSzcskHt}, let $x,y\in L$. Denote $x\wedge y$ and $x\vee y$ by $u$ and $v$, respectively.
We have that $\phi(x)=(x_1,\dots,x_n)$ and $\phi(y)=(y_1,\dots,y_n)$. Here $y_i$ is the largest element of $\Ep i\cap \ideal y$, and analogous notation applies for $\phi(u)$ and $\phi(v)$. Since the lattice operations in the direct product $G$ are computed componentwise, we only need to show that, for every $i\in \set{1,\dots,n}$, $x_i\wedge y_i=u_i$ and 
$x_i\vee y_i=v_i$. In fact, we only need to show that $x_i\wedge y_i\leq u_i$ and $x_i\vee y_i\geq v_i$ since the converse inequalities follow from the fact that $\phi$ is clearly order-preserving. Since $x_i$ and $y_i$ belong to the same chain, $\Ep i$, these two elements are comparable. They play a symmetrical role, whence we can assume that 
$x_i\leq y_i$. Thus, the equalities $x_i=x_i\wedge y_i$ and $y_i=x_i\vee y_i$ reduce our task  to show that  $x_i\leq u_i$ and $y_i\geq v_i$. Since $x_i\in \Ep i\cap \ideal x$
and $x_i\leq y_i$ yields that $x_i\in \Ep i\cap \ideal y$, 
we have that $x_i\in \Ep i\cap \ideal x\cap \ideal y=\Ep i\cap \ideal (x\wedge y)=\Ep i\cap \ideal u$. Taking into account that $u_i$ is the largest element of $\Ep i\cap \ideal u$, the required inequality $x_i\leq u_i$ follows. 
It belongs to the folklore of lattice theory (and it occurs in the last paragraph of the proof of Theorem 107 in
Gr\"atzer \cite{r:Gr-LTFound})  that
\begin{equation}
\parbox{8cm}{if $D$ is a finite distributive lattice, $t\in\Nplu$, $p\in\Jir D$, $q_1,\dots,q_t\in D$, and $p\leq q_1\vee\dots\vee q_t$, then there is an $i\in\set{1,\dots,t}$ such that $p\leq q_i$.}
\label{eq:fLktnhLdSrbK}
\end{equation}
Indeed, if the premise of \eqref{eq:fLktnhLdSrbK} holds, then 
$p=p\wedge(q_1\vee\dots\vee q_t)=(p\wedge q_1)\vee\dots\vee (p\wedge q_t)$ and $p\in \Jir D$ yield that $p=p\wedge q_i$ for some $i$, implying the required $p\leq q_i$.  
Resuming our argument for $\phi$,  we know that $v_i\leq v=x\vee y$ and $v_i\in \Ep i\subseteq \Jir L$. Hence
\eqref{eq:fLktnhLdSrbK} gives that $v_i\leq x$ or $v_i\leq y$.   If $v_i\leq x$, then the definition of $x_i$ yields that $v_i\leq x_i$, whence $v_i\leq y_i$. If $v_i\leq y$, then the definition of $y_i$ immediately yields that $v_i\leq y_i$. So the required $y_i\geq v_i$ holds in both cases, and we have shown that $\phi$ is a lattice homomorphism. 

Next, we claim that for each $x\in L$, 
\begin{equation}
x = x_1\vee\dots\vee x_n.
\label{eq:GmrSlrkLrSztbrkK}
\end{equation}
By finiteness, there is a subset $H$ of $\Jir L$ such that $x=\bigvee H$. For each $h\in H$, \eqref{pbx:szWrknCSkBrGlnD} and $\Jir L=C_1\cup\dots\cup C_n$ yield an $i\in\set{1,\dots,n}$ such that $h\in \Ep i$. Then we have that $h\in\Ep i\cap\ideal x$, whereby $h\leq x_i\leq x_1\vee\dots\vee x_n$. 
Since this holds for all $h\in H$, we have that $x=\bigvee H\leq  x_1\vee\dots\vee x_n$. The converse inequality is trivial, and we conclude \eqref{eq:GmrSlrkLrSztbrkK}.
Clearly, \eqref{eq:GmrSlrkLrSztbrkK} implies the injectivity of $\phi$. Thus, we have shown \eqref{pbx:zgmSzcskHt}.

Finally, \eqref{eqtxt:sHmlFksRb}, \eqref{pbx:zgmSzcskHt},   and Lemma~\ref{lemmaKzTfPsz}\eqref{lemmaKzTfPszb} imply that $f$ is a \covzo, completing the proof of  Lemma~\ref{lemma:nLskPcjktNLTn}.
\end{proof}

\begin{lemma}\label{lemma:notbOOle}
If $n\in\Nplu$, $L$ is an $n$-dimensional grid, but $L$ is not a boolean lattice, then $L$ is a sublattice of an $(n+1)$-dimensional grid $K$ such that $K$ and $L$ are of the same length.
\end{lemma}

\begin{proof} By the assumption, $L=C_1\times\dots\times C_n$ such that $C_1$, \dots, $C_n$ are nontrivial chains and at least one of them consists of at least three elements. Up to isomorphism, the order of the direct factors is irrelevant, whereby we can assume that $|C_1|\geq 3$. Let $q$ be the unique coatom of $C_1$. Then $0<q\prec 1$ in $C_1$ and $E_0:=\filter q=\set{q,1}$ is a two-element subchain of $C_1$. 
The subchain $E_1:=\ideal q$ is still a nontrivial chain. 
Define $K:=E_0\times E_1\times C_2\times \dots\times C_n$. 
It is an $(n+1)$-dimensional grid. Since $\Jir L$ consists of the vectors with exactly one nonzero component and similarly for $\Jir K$,
$|\Jir L|=(|C_1|-1)+(|C_2|-1)+\dots+(|C_n|-1)= (|E_0|-1)+(|E_1|-1)+(|C_2|-1)+\dots+(|C_n|-1)=|\Jir K|$. Hence, \eqref{pbx:lJrsBlTlgR} gives that $L$ and $K$ are of the same length. We are going to show that $L$ can be embedded into $K$. 

Instead of defining an injective homomorphism $L\to K$ and verifying its properties in a tedious way, recall the following. If $H_1$ and $H_2$ are lattices, $F_1$ is a filter of $H_1$, $I_2$ is an ideal of $H_2$, and $\psi\colon F_1\to I_2$ is a lattice isomorphism, then the quintuplet $(H_1,H_2,F_1,I_2,\psi)$ uniquely determines a lattice $H$ by identifying $x$ with $\psi(x)$, for all $x\in F_1$, in  $H_1\cup H_2$. This $H$ is the well-known \emph{Hall--Dilworth gluing} of $H_1$ and $H_2$  or, to be more precise, the  Hall--Dilworth gluing determined by the quintuplet; see, for example, Gr\"atzer~\cite[Lemma 298]{r:Gr-LTFound} for more details. Furthermore, it is also well known, see  Gr\"atzer~\cite[Lemma 299]{r:Gr-LTFound}, that
\begin{equation}
\parbox{8.2cm}{if $M$ is  a lattice, $M_1$ is an ideal of $M$, $M_2$ is a filter of $M$, and $T:=M_1\cap M_2\neq \emptyset$, then $M_1\cup M_2$ is a sublattice of $M$ and $M$ is isomorphic to the Hall--Dilworth gluing determined by $(M_1,M_2, T, T,\id T)$,}
\label{pbx:wGhrSgcSrgsJ}
\end{equation}
where $\id T\colon T\to T$ is the \emph{identity map} defined by $x\mapsto x$.  

In the rest of \emph{this} proof,  $\vec 0$ and $\vec 1$ will stand for $(0_{C_2},\dots,0_{C_n})\in C_2\times\dots \times C_n$ and $(1_{C_2},\dots,1_{C_n})\in C_2\times\dots \times C_n$, respectively.  
In $L$, we let $I_L:=\ideal{(q,\vec 1\,)}$, $F_L:=\filter{(q,\vec 0\,)}$, and $T_L:=I_L\cap F_L=[ (q,\vec 0), (q,\vec 1) ]$. 
In $K$, we let $I_K:=\ideal{(q,q,\vec 1\,)}$; remember that the first $q$ here is the least element of $E_0$ while the second $q$ is the largest element of $E_1$. Still in $K$, we also let $F_K:=\filter{(q,q,\vec 0\,)}$ 
and $T_K:=I_K\cap F_K=[ (q,q,\vec 0\,), (q,q,\vec 1\,) ]$. 
Clearly, the map $\rho\colon I_L\to, I_K$ defined by $(x,\vec y\,)\mapsto (q,x,\vec y\,)$ is an isomorphism. Let
\begin{equation*}
\tau\colon F_L\to F_K\text{ be defined by } (x,\vec y\,)\mapsto (x,q,\vec y);
\end{equation*}
it is also an isomorphism. We have to check that each of the restrictions $\restrict \rho{T_L}$ and $\restrict \tau{T_L}$ are the same maps and they are $T_L\to T_K$ isomorphisms. But this is clear since $T_L=\set{ (q,\vec y\,): \vec y\in C_2\times\dots \times C_n}$ and $T_K=\set{ (q,q,\vec y\,): \vec y\in C_2\times\dots \times C_n}$. 
Hence, it follows from \eqref{pbx:wGhrSgcSrgsJ} that 
$I_K\cup F_K$ is a sublattice of $K$. It also follows from \eqref{pbx:wGhrSgcSrgsJ} that 
$L$, which is the Hall--Dilworth gluing determined by  $(I_L, F_L, T_L, T_L, \id{T_L})$, is isomorphic to this sublattice. Therefore, after replacing $K$ by an isomorphic copy if necessary, we conclude that $L$ is a sublattice of $K$, proving Lemma~\ref{lemma:notbOOle}.
\end{proof}

\subsection{Our lemmas at work}
Armed with our lemmas, we are ready to prove the main theorem of the paper and Corollaries~\ref{cor:DfnVGs}--\ref{cor:krSzwzlvnlZvl}. First, we disregard Corollary~\ref{cor:krSzwzlvnlZvl} in the proof below.

\begin{proof}[Proof of Theorem \ref{thmdst} and Corollaries~\ref{cor:DfnVGs}--\ref{cor-kvsnLptv}]
It follows from Proposition~\ref{prop} that 
\begin{equation}
\parbox{7.4cm}{\eqref{thmdsta},  \eqref{thmdstb}, and  \eqref{thmdstc} are equivalent in each of 
Theorem~\ref{thmdst}, Corollary~\ref{cor:DfnVGs}, and Corollary~\ref{cor-kvsnLptv}.}
\label{pbx:sChsWsZjMSvZlf}
\end{equation}

Next, we are going to prove that for $n\in\Nplu$ and a  \emph{finite} distributive lattice $D$,
\begin{align}
\parbox{8.4cm}{if $D\in \Dncovzo n$ is an absolute $\Hop$-retract for $\Dncovzo n$, then $D$ is boolean or $D$ is an $n$-dimensional grid;}
\label{pbx:whWbznTnka}
\\
\parbox{8.4cm}{if $D$ is boolean, then $D$  is an absolute retract for $\Dall$.}
\label{pbx:whWbznTnkb}
\\
\parbox{8.4cm}{if $D$ is an $n$-dimensional grid, then $D$  is an absolute retract for $\Dnfin n$;}
\label{pbx:whWbznTnkc}
\end{align}

To prove \eqref{pbx:whWbznTnka}, assume that $n\in\Nplu$ and  $D\in \Dncovzo n$ is an absolute $\Hop$-retract for $\Dncovzo n$.
For the sake of contradiction, suppose that $D$ neither boolean  nor it is an $n$-dimensional grid. 
The first task in the proof is to find a \emph{proper} $\Dncovzo n$-extension $K$ of $D$.
Let $k:=\dim D$. By Lemma~\ref{lemma:nLskPcjktNLTn}, $D$ has a $\Dncovzo n$-extension $L$ such that $L$ is a $k$-dimensional grid. There are three cases depending on $k$ and $L$. 

First, assume that $k<n$ and $L$ is boolean. Then $D\neq L$ since $D$ is not boolean. So if we let $K:=L$, then 
\begin{equation}
\parbox{6.0cm}{$K\in\Dncovzo n$, $K\neq D$, and $K$ is a $\Dncovzo n$-extension of $D$.
}\label{pbx:wmRsznkvTDfrKs}
\end{equation}

Second, assume that $k<n$ and $L$ is not boolean. Then \eqref{pbx:lnPrsHm} gives that $\length L=\length D$. Lemma~\ref{lemma:notbOOle} allows us to take a $(k+1)$-dimensional 
grid $K$ such that $\length K=\length L$ and $L$ is a sublattice of $K$. So $D$ is a sublattice of $K$ and $\length D=\length K$. 
Hence if we apply  Lemma~\ref{lemmaKzTfPsz}\eqref{lemmaKzTfPszb}
to the map $D\to K$ defined by $x\mapsto x$ and take  $\dim K= k+1\leq n$ into account, we obtain that $K$ is a $\Dncovzo n$-extension of $D$. Since $\dim D=k\neq \dim K$, we have that $D\neq K$ and so \eqref{pbx:wmRsznkvTDfrKs} holds again.

Third, assume that $k=n$, that is,   $\dim D=n$. Then, by  Lemma~\ref{lemma:nLskPcjktNLTn}, $D$ has a $\Dncovzo n$-extension $K$ such that $K$ is an $n$-dimensional grid.  Since we have assumed that $D$ is not an $n$-dimensional grid, \eqref{pbx:wmRsznkvTDfrKs} holds again.

We have seen that, in each of the three possible cases, \eqref{pbx:wmRsznkvTDfrKs}  holds. Since  $D\in
\Dncovzo n$ was assumed to be an absolute $\Hop$-retract for $\Dncovzo n$, there exists a retraction $f\colon K\to D$.
We know from \eqref{pbx:wmRsznkvTDfrKs} that the map $D\to K$ defined by $x\mapsto x$ is a \embcovzo. Hence $f$ is an isomorphism by Lemma~\ref{lemmaKzTfPsz}\eqref{lemmaKzTfPszc}, whereby $|K|=|D|$. This contradicts the fact that $D$ is a proper sublattice of $K$ by \eqref{pbx:wmRsznkvTDfrKs}, 
and we have proved \eqref{pbx:whWbznTnka}.

To prove \eqref{pbx:whWbznTnkb}, assume that  a finite boolean lattice $D$ is a sublattice of a not necessarily finite distributive lattice $K$. We are going to show that there exists a retraction $K\to D$. Since this is trivial if $D$ is a singleton, we can assume that $|D|>1$.  Let $n:=\dim D$.  Combining \eqref{eqtxt:dimDwJD} and \eqref{pbx:lJrsBlTlgR} and taking into account that the join-irreducible elements of a finite boolean lattice  are exactly its atoms, it follows that $D$ has exactly $n$ atoms and it is of length $n$. Hence, we can take a maximal chain
$C=\set{0=c_0, c_1,\dots, c_{n-1}, c_n=1}$ in $D$ such that $c_{i-1}\prec c_i$ for $i\in\set{1,\dots,n}$. For   $i\in\set{1,\dots,n}$, the Prime Ideal Theorem allows us to pick a prime ideal $I_i$ of $K$ such that $c_{i-1}\in I_i$ but $c_i\notin I_i$. Since $I_i$ is a prime ideal, the partition $\set{I_i,K\setminus I_i}$ determines a congruence $\Theta_i$ of $K$. This congruence \emph{separates} $c_{i-1}$ and $c_i$, that is, $(c_{i-1}, c_i)\notin \Theta_i$. Let $\Theta:=\bigcap\set{\Theta_i: i\in\set{1,\dots n}}$. Now $\Theta$  is a congruence of $K$ and its restriction $\restrict \Theta C$ is a congruence of the sublattice $C$. We claim that  $\restrict \Theta C=\diag C$; suppose the contrary. We know from the folklore that any congruence of a finite lattice is determined by the covering pairs it collapses, whence $(c_{i-1},c_i)\in\restrict \Theta C$ for some $i\in\set{1,\dots,n}$. But then $(c_{i-1},c_i)\in\restrict \Theta C \subseteq   \Theta\subseteq \Theta_i$, contradicting the fact that $\Theta_i$ separates $c_{i-1}$ and $c_i$. This shows that $\restrict \Theta C=\diag C$. Therefore, it follows from \eqref{pbx:gRnNt} and  
 $\restrict \Theta C=  \restrict{(\restrict \Theta D)} C$ that
\begin{equation}
\restrict \Theta D = \diag D.
\label{eq:tGHjszszDr}
\end{equation}
Observe that
\begin{equation}
\parbox{8.7cm}{if $\alpha$ and $\beta$ are congruences of a not necessarily finite lattice, $\alpha$ has exactly $m\in\Nplu$ blocks, and $\beta$ has exactly $n\in\Nplu$ blocks, then $\alpha\cap\beta$ has at most $mn$ blocks.}
\label{pbx:mkhszTvn}
\end{equation}
Indeed, each of the $m$ $\alpha$-blocks is cut into at most $n$ pieces by $\beta$. Since $\Theta_i$ has only two blocks, it follows from \eqref{pbx:mkhszTvn} that $\Theta$ has at most $2^n$ blocks. But the elements of $D$ belong to pairwise different $\Theta$-blocks by \eqref{eq:tGHjszszDr}, whereby $\Theta$ has exactly $2^n=|D|$ blocks.
Next, we define a map
\begin{equation}
f\colon K\to D\text{ by the rule }f(x)=d\in D\iff (x,d)\in\Theta.
\label{eq:bRhKlkPzLcsb}
\end{equation}
For later reference, we note that 
\begin{equation}
\parbox{7.7cm}{to show that $f$ in \eqref{eq:bRhKlkPzLcsb} is a retraction, we will only use that $K$ is a lattice,  $D$ is finite a sublattice of $K$, $\Theta$ has exactly $|D|$ blocks, and $\restrict \Theta D=\diag D$.}
\label{bpx:mzRdfZTrskgGN}
\end{equation}
Since $\Theta$ has exactly $2^n$ blocks,  $|D|=2^n$ and \eqref{eq:tGHjszszDr} guarantee the properties mentioned in  \eqref{bpx:mzRdfZTrskgGN}. The equality $\restrict \Theta D=\diag D$ yields that for each $x\in K$, there is at most one $d$ in \eqref{eq:bRhKlkPzLcsb}. 
If there was an $x\in K$ with its $\Theta$-block $x/\Theta$ disjoint from $D$, then $\Theta$ would have more than $|D|$-blocks since $x/\Theta$ would be different from the pairwise distinct blocks of the elements of $D$. Thus, for each $x\in K$, there is exactly one $d\in D$ with $(x,d)\in\Theta$, whereby 
\eqref{eq:bRhKlkPzLcsb} defines a map, indeed. If $f(x_1)=d_1$ and $f(x_2)=d_2$, then $(x_1,d_1)\in\Theta$ and $(x_2,d_2)\in\Theta$ yield that $(x_1\vee x_2,d_1\vee d_2)\in\Theta$, whence $f(x_1\vee x_2)=d_1\vee d_2\in D$. The same holds for meets, and so $f$ is a homomorphism. By the reflexivity of $\Theta$, $f(d)=d$ for all $d\in D$. Thus, $f$ is a retraction, proving \eqref{pbx:whWbznTnkb}.

Next, to prove \eqref{pbx:whWbznTnkc}, we begin with focusing on its simplest particular case.  Namely, we claim that 
\begin{equation} 
\parbox{5cm}{If $E$ is a subchain of a finite chain $C$, then
$E$ is a retract of $C$.}
\label{pbx:lNcSzrLrtRct}
\end{equation}
To see this, let $E=\set{ e_1, e_2,\dots e_{k}}$ such that 
$e_1<e_2<\dots<e_{k}$. Understanding the principal ideals below in $C$, it is trivial that the equivalence $\Theta$ with blocks $\ideal{e_1}$, $\ideal{e_2}\setminus \ideal{e_1}$, \dots, $\ideal{e_{k-1}}\setminus \ideal{e_{k-2}}$, $C\setminus \ideal{e_{k-1}}$ is a congruence of $C$. Since $\restrict \Theta E=\diag E$ and $\Theta$ has $|E|$ blocks, \eqref{bpx:mzRdfZTrskgGN} implies \eqref{pbx:lNcSzrLrtRct}.

Armed with \eqref{pbx:lNcSzrLrtRct}, assume that $n\in\Nplu$, $D$ is an $n$-dimensional grid, 
 $L\in \Dnfin n$, and $D$ is a sublattice of $L$. We are going to find a retraction $L\to D$. It follows from Milner and Pouzet \cite{milnerpouzet}, see Subsection~\ref{subsect:Dfdim} of the present paper, that $\dim D\leq\dim L$. 
Combining this inequality with $n=\dim D$ and $L\in \Dnfin n$, we obtain that $\dim L=n$. Hence, by 
 Lemma~\ref{lemma:nLskPcjktNLTn}, there is a \embcovzo{} of $L$ into an $n$-dimensional grid $K$. 
Then $D$ is a sublattice of $K$, and both $D$ and $K$ are $n$-dimensional grids. Let $C_1$, \dots, $C_n$ be the canonical chains of $K$. By Lemma \ref{lemma:gRd}, these canonical chains have nontrivial subchains $E_1$, \dots, $E_n$, respectively, such that \eqref{eq:dcEdjKwmC} holds (for $D$ in place of $L$). For $i\in\set{1,\dots,n}$, $\pi_i\colon K\to C_i$ defined by $x\mapsto \cj x i$ is a homomorphism by Lemma~\ref{lemma:cnChgR}\eqref{lemma:cnChgRc}. 
Since $\cj x i=0\vee\dots\vee 0\vee \cj x i \vee 0\vee\dots\vee 0$ (where $\cj x i$ is the $i$-th joinand on the right), the uniqueness  of the canonical form \eqref{eq:cFnfRmsG} gives that $\cj{(\cj x i)}i=\cj x i$. Hence, $\pi_i$ acts identically on $C_i$ and so $\pi_i$ is a retraction. Using \eqref{pbx:lNcSzrLrtRct}, we can take a retraction $g_i\colon C_i\to E_i$. Clearly, the composite map $f_i:=g_i\circ\pi_i$ is a retraction $K\to E_i$. 
For $x\in D$ , \eqref{eq:dcEdjKwmC} gives that $\cj x i\in E_i$. Hence, for $x\in D$ and $i\in\set{1,\dots,n}$, 
\begin{equation}
f_i(x)=g_i(\pi_i(x))=g_i(\cj x i) =\cj x i.
\label{eq:vszmsNgTPrsTb}
\end{equation}

Let $\Theta_i$ be the kernel of $f_i$. Since $f_i$, as any retraction, is surjective, $\Theta_i$ has exactly $|E_i|$ blocks. Therefore, if we let $\Theta:=\bigcap\set{\Theta_i: i\in\set{1,\dots,n}}$, then $\Theta$ is a congruence of $K$ with at most $|E_1|\cdots |E_n|=|D|$ blocks by \eqref{pbx:mkhszTvn}. On the other hand, if $(x,y)\in \Theta$ holds for $x,y\in D$, then $(x,y)\in \Theta_i$ and \eqref{eq:vszmsNgTPrsTb} give that $\cj x i = f_i(x)=f_i(y)=\cj y i$ for all $i\in\set{1,\dots,n}$, whence it follows from \eqref{eq:cFnfRmsG} that $x=y$. This means that $\restrict \Theta D=\diag D$. Thus, $\Theta$ has at least $|D|$ blocks, and  obtain that $\Theta$ has exactly $|D|$-blocks.
Therefore, \eqref{eq:bRhKlkPzLcsb} and \eqref{bpx:mzRdfZTrskgGN} imply that there is a retraction $f\colon K\to D$. Since the restriction $\restrict f L\colon L\to D$, defined by $x\mapsto f(x)$, is clearly a retraction, we have shown the existence of a retraction $L\to D$, as required.
 This completes the proof \eqref{pbx:whWbznTnkc}.

For categories $\alg X$ and $\alg Y$, we say that $\alg X$ is a 
\emph{subcategory} of $\alg Y$ if every object of $\alg X$ is an object of $\alg Y$ and every morphism  of $\alg X$ is a morphism of $\alg Y$. The following two observations are trivial.
\begin{align}
\parbox{9cm}{If $\alg X$ and $\alg Y$ are categories of lattices such that  $\alg X$ is a subcategory of $\alg Y$ and a lattice $L\in\alg X$ is an absolute $\Hop$-retract for $\alg Y$, then $L$ is also an absolute $\Hop$-retract also for $\alg X$.}\label{pbx-hdNrzKfmCka}\\
\parbox{9cm}{An absolute retract for a category of lattices is also an absolute $\Hop$-retract for that category.}
\label{pbx-hdNrzKfmCkb}
\end{align}

For Theorem~\ref{thmdst}, in virtue of \eqref{pbx:sChsWsZjMSvZlf}, it suffices to show that \ref{thmdst}\eqref{thmdstc} and \ref{thmdst}\eqref{thmdstd} are equivalent conditions. Assume \ref{thmdst}\eqref{thmdstc}, that is, let $D\in \Dnfin n$ be an absolute retract of $\Dnfin n$. By \eqref{pbx-hdNrzKfmCkb}, $D$ is an absolute $\Hop$-retract for $\Dnfin n$.  If  $n\in\Nplu$, then  \eqref{pbx-hdNrzKfmCka} gives that $D$ is an absolute $\Hop$-retract of $\Dncovzo n$, whereby  \eqref{pbx:whWbznTnka} yields \ref{thmdst}\eqref{thmdstd}, as required. So we can assume that  $n=\omega$. Denote $\dim D$ by $k$. Then $D\in\Dncovzo{k+1}$, and    \eqref{pbx-hdNrzKfmCka} gives that $D$ is an absolute $\Hop$-retract of $\Dncovzo {k+1}$. By \eqref{pbx:whWbznTnka}, $D$ is boolean or $D$ is a $(k+1)$-dimensional grid. The second alternative is ruled out by $\dim D=k$, whence  \ref{thmdst}\eqref{thmdstd} holds for $D$. We have seen that   \ref{thmdst}\eqref{thmdstc} implies  \ref{thmdst}\eqref{thmdstd}.

Conversely, assume that \ref{thmdst}\eqref{thmdstd} holds for  finite distributive lattice $D$.  If $D$ is boolean, then it is an absolute $\Hop$-retract for $\Dall$ by \eqref{pbx:whWbznTnkb} and \eqref{pbx:rmDhwhspKlQB}, whereby \ref{thmdst}\eqref{thmdstc} holds for $D$ by \eqref{pbx-hdNrzKfmCka} and \eqref{pbx:rmDhwhspKlQB}. If $D$ is an $n$-dimensional grid, then \eqref{pbx:whWbznTnkc} immediately implies that  \ref{thmdst}\eqref{thmdstc} holds for $D$. We have proved Theorem~\ref{thmdst}.  

Next, assume that a finite distributive lattice $D$ satisfies \ref{cor:DfnVGs}\eqref{cor:DfnVGsc}, that is, $D$ is an absolute retract for $\Dall$. Let $n:=\dim D$. Combining \eqref{pbx:rmDhwhspKlQB} and \eqref{pbx-hdNrzKfmCka}, we obtain that 
$D$ is an absolute $\Hop$-retract for $\Dncovzo{n+1}$. 
By \eqref{pbx:whWbznTnka}, $D$ is boolean or it is an $(n+1)$-dimensional grid. But $\dim D=n$ excludes the second alternative, and we conclude that $D$ satisfies  \ref{cor:DfnVGs}\eqref{cor:DfnVGsd}. This shows implication \ref{cor:DfnVGs}\eqref{cor:DfnVGsc} $\Rightarrow$ \ref{cor:DfnVGs}\eqref{cor:DfnVGsd}. Since the converse implication is just \eqref{pbx:whWbznTnkb}, we have verified Corollary~\ref{cor:DfnVGs}.

For Corollary~\ref{cor-kvsnLptv},  observe that \ref{cor-kvsnLptv}\eqref{cor-kvsnLptvd} $\Rightarrow$ \ref{cor-kvsnLptv}\eqref{cor-kvsnLptvc} by   \eqref{pbx-hdNrzKfmCkb}
while  \ref{cor-kvsnLptv}\eqref{cor-kvsnLptvc}  $\Rightarrow$ 
\ref{cor-kvsnLptv}\eqref{cor-kvsnLptve} is just \eqref{pbx:whWbznTnka}. So we only need to show that  \ref{cor-kvsnLptv}\eqref{cor-kvsnLptve}  $\Rightarrow$ 
\ref{cor-kvsnLptv}\eqref{cor-kvsnLptvd}. Assume that  $D\in\Dncovzo n$  satisfies \ref{cor-kvsnLptv}\eqref{cor-kvsnLptve}. 
There are two cases depending on whether $D$ is boolean or not. First, assume that $D\in \Dncovzo n$ is boolean. Then $D$ is an absolute $\Hop$-retract for $\Dall$ by \eqref{pbx:whWbznTnkb} and  \eqref{pbx:rmDhwhspKlQB}. We obtain from  \eqref{pbx-hdNrzKfmCka} that $D$ is as absolute $\Hop$-retract for $\Dncovzo n$. Second, assume that $D$ is an $n$-dimensional grid. Then $D$ is an absolute $\Hop$-retract for $\Dncovzo n$ by
 \eqref{pbx:whWbznTnkc} and \eqref{pbx-hdNrzKfmCkb}.
So in both cases,  $D$ is an absolute $\Hop$-retract for $\Dncovzo n$.  Let $K$ be a $\Dncovzo n$-extension of $D$. Since $D$ is an absolute $\Hop$-retract for $\Dncovzo n$, there exists a retraction $f\colon K\to D$. By Lemma~\ref{lemmaKzTfPsz}\eqref{lemmaKzTfPszc}, $f$ is a morphism of $\Dncovzo n$. This shows that \ref{cor-kvsnLptv}\eqref{cor-kvsnLptvd} holds for $D$, as required.  
We have verified Corollary~\ref{cor-kvsnLptv}, and the proof is complete.
\end{proof}

\begin{proof}[Proof of Corollary~\ref{cor:krSzwzlvnlZvl}]
By Proposition 5.2 of Kelly and Rival~\cite{kellyrival},  a finite lattice is planar if and only if its order dimension is at most 2. Hence, the class of planar distributive lattices is $\Dnfin 2$. 
Thus, the equivalence of \ref{cor:krSzwzlvnlZvl}\eqref{cor:krSzwzlvnlZvla}
and \ref{cor:krSzwzlvnlZvl}\eqref{cor:krSzwzlvnlZvld} follows from Theorem~\ref{thmdst} while the equivalence of 
\ref{cor:krSzwzlvnlZvl}\eqref{cor:krSzwzlvnlZvlb}, \ref{cor:krSzwzlvnlZvl}\eqref{cor:krSzwzlvnlZvlc}, and \ref{cor:krSzwzlvnlZvl}\eqref{cor:krSzwzlvnlZvld} is a consequence of Corollary~\ref{cor-kvsnLptv}.
\end{proof}

\color{black}

\end{document}